\documentclass[12pt]{article}
\input epsf.tex

\usepackage{amsmath}
\usepackage{amsthm}
\usepackage{amsfonts}
\usepackage{amssymb}
\usepackage{graphicx}
\usepackage{latexsym}

\usepackage{amsmath,amsthm,amsfonts,amssymb}

\usepackage{epsfig}

\usepackage[]{xy}

\theoremstyle{plain}
\newtheorem{thm}{Theorem}[section]
\newtheorem{prop}[thm]{Proposition}
\newtheorem{lem}[thm]{Lemma}
\newtheorem{cor}[thm]{Corollary}

\theoremstyle{definition}
\newtheorem{defn}{Definition}
\theoremstyle{remark}
\newtheorem{remark}{Remark}

\advance \topmargin by -\headheight
\advance \topmargin by -\headsep
\evensidemargin \oddsidemargin
\def\cB{{\mathcal B}}

\def\c{\curvearrowright}

\def\cC{{\mathcal C}}

\def\bE{{\mathbb E}}

\def\ed{\textrm{End}}

\def\cF{{\mathcal F}}\def\sF{{\mathcal F}}

\def\sG{{\mathcal G}}
\def\cG{{\mathcal G}}

\def\N{{\mathbb N}}

\def\cN{{\mathcal{N}}}

\def\P{{\bf P}}

\def\sP{{\mathcal P}}
\def\cP{{\mathcal P}}

\def\past{{\textrm{Past}}}

\def\R{{\mathbb R}}

\def\bX{{\bar X}}

\def\chix{{\raise.5ex\hbox{$\chi$}}}

\def\bY{{\bar Y}}

\def\Z{{\mathbb Z}}

\begin{document}
\title{Nonabelian free group actions: Markov processes, the Abramov-Rohlin formula and Yuzvinskii's formula}
\author{Lewis Bowen\footnote{email:lpbowen@math.hawaii.edu} \\ University of Hawaii}
\begin{abstract}
This paper introduces Markov chains and processes over nonabelian free groups and semigroups. We prove a formula for the $f$-invariant of a Markov chain over a free group in terms of transition matrices that parallels the classical formula for the entropy a Markov chain. Applications include free group analogues of the Abramov-Rohlin formula for skew-product actions and Yuzvinskii's addition formula for algebraic actions. 
\end{abstract}
\maketitle

\noindent


\section{Introduction}

A (classical) Markov chain is an $\N$ or $\Z$-indexed family of random variables $(X_i)$ each taking values in a set $K$ (called the state space) and satisfying the following condition: for any $i \in \N$ or $\Z$ and $k_{i+1}, k_i, \ldots \in K$
 $$Pr(X_{i+1} = k_{i+1} | X_i = k_i) = Pr(X_{i+1}=k_{i+1}~|~X_i=k_i, X_{i-1}=k_{i-1},\ldots)=M_{k_{i},k_{i+1}}$$
 where $M$ is a fixed $K \times K$ matrix called the transition matrix. We will always assume that $K$ is at most countable and has the discrete topology.

These objects can be viewed from an ergodic theory perspective as follows. Let $K^G$ denote the set of all functions $G \to K$ where $G$ equals $\Z$ or $\N$. Let $\mu$ be the probability measure on $K^G$ defined by setting $\mu(E)$ equal to the probability that the Markov chain $(X_i)$ (considered as a function from $G$ to $K$) is in $E$. Let $\alpha$ be the ``time $0$ partition'' defined by $\alpha=\{A_k~:~k\in K\}$ where $A_k=\{x:G\to K~:~x(0)=k\}$. Let $\sigma:K^G \to K^G$ be the shift-operator, defined by $\sigma(x)(n) = x(n+1)$. We call the quadruple $(\sigma,K^G,\mu,\alpha)$ a {\bf Markov process}.

When $\mu$ is shift-invariant, it satisfies several nice properties. First, the entropy rate $h(\sigma, \mu)$ equals $H(\alpha | \sigma^{-1}\alpha)$ where $H(\cdot)$ is Shannon's entropy (see \S \ref{sub:classical} for the definition). In fact, this property characterizes Markov processes. Second, the space of all $n$-step Markov-processes (which are generalizations of the above) is dense in the space of shift-invariant Borel probability measures on $K^G$ with the weak* topology. 

The purpose of this paper is to build an analogous theory when $G$ is a free group or semigroup and entropy is replaced with the $f$-invariant. The latter is a measure-conjugacy invariant that generalizes Kolmogorov-Sinai entropy. It was introduced in [Bo08a]. For the reader's convenience, the $f$-invariant is defined in \S \ref{sub:f} and (most of) the proof that it is a measure-conjugacy is recalled in \S \ref{sec:alternative}.

The definition of Markov chain over a free group is similar to the definition of tree-indexed Markov chain (see [Pe95] and the references therein). The primary novelty here is that we study ergodic-theoretic aspects and especially relationships with entropy theory. 

The notion of a Markov chain is connected with the notion of a ``past''. For example, the past of an element $n\in \Z$ is the set $\Z \cap (-\infty,n)$. A stationary Markov chain $(X_i)_{i \in \Z}$ is characterized by the property that the distribution of $X_i$ conditioned on $X_{i-1}$ (its immediate past) equals the distribution of $X_i$ conditioned on $X_j$ for all $j$ in the past of $i$ and that this conditional distribution is independent of $i$.

If $G=\langle s_1,\ldots,s_r\rangle$ is a free group then every element $g\in G$ has $2r$ different ``pasts'', corresponding to the given generators. For example, the past of an element $g$ in the direction of $s\in \{s_1^{\pm 1}, \ldots, s_r^{\pm 1}\}$ is the set of all elements of the form $fsg$ where $f\in G$, $|fsg|=|f|+|sg|$ and $|\cdot|$ denotes the word-metric on $G$.  A Markov chain over $G$ is a $G$-indexed family of random variables $(X_g)_{g\in G}$ such that the distribution of $X_g$ conditioned on $X_{sg}$ equals the distribution of $X_g$ conditioned on $X_h$ for all $h$ in the past of $g$ in the $s$-direction and that this conditional distribution is independent of $g$. 

To view this from this ergodic theory perspective, for $g \in G$, let $T_g:K^G \to K^G$ be the shift-operator defined by $T_g(x)(f)=x(fg)$. Let $\mu$ be the measure on $K^G$ defined by $\mu(E)$ equals the probability that the Markov chain $(X_g)_{g\in G}$ considered as a function from $G$ to $K$ is in $E$. Let $\alpha$ be the ``time $e$-partition'': $\alpha=\{A_k~:~k\in K\}$ where $A_k=\{x\in K^G~:~x(e)=k\}$. The action $G \c^T (K^G,\mu)$ with the partition $\alpha$ is a Markov process. A more general and precise definition is in \S \ref{sec:markovprocesses}.

In the classical case, stationary Markov chains can be easily constructed in terms of transition matrices and stationary vectors. We show that there is an analogous construction in the case of free groups. This should be useful to the study of the classification problem for dynamical systems over free groups up to isomorphism. For example, it is shown in \S \ref{sec:example3} that there is mixing Markov chain that is not isomorphic to any Bernoulli shift. This contrasts with the Friedman-Ornstein theorem [FO70] that every mixing Markov chain over the integers is isomorphic to a Bernoulli shift. We also exhibit examples of Markov chains related to well-studied problems in the theory of random regular graphs.


Assume now that $(X_g)_{g\in G}$ is a stationary Markov chain. This implies $\mu$ is shift-invariant. We will show that the $f$-invariant of the system $G \c^T (K^G,\mu)$ equals $F(\mu,\alpha):= (1-2r)H(\alpha) + \sum_{i=1}^r H(\alpha \vee T_{s_i}^{-1}\alpha)$. Indeed, this condition characterizes Markov processes. Since $f(\mu,\alpha) \le F(\mu,\alpha)$ holds in general, it follows that for every shift-invariant Borel probability measure $\nu$ on $K^G$ that equals $\mu$ on the partitions $\alpha \vee T_{s_i}^{-1}\alpha$, $f(\nu,\alpha) \le f(\mu,\alpha)$ with equality if and only if $\mu=\nu$. In brief: the $f$-invariant is uniquely maximized on Markov chains. Moreover, there is a precise sense in which every process over $G$ can be approximated by a sequence of ``higher-step'' Markov processes. These tools are used to prove analogues of two classical results: the Abramov-Rohlin formula and Yuzvinskii's addition formula. To explain these results precisely, let us review the definitions of entropy and the $f$-invariant next.

\subsection{Classical results}\label{sub:classical}

Let $(X,\cB,\mu)$ be a probability space. Let $T:X \to X$ be a measure-preserving transformation. We use $\alpha, \beta$ to denote measurable partitions of $X$ into at most countable many subsets. The join of $\alpha$ and $\beta$ is their common refinement, defined by $\alpha \vee \beta :=\{A \cap B~:~ A\in\alpha, B \in \beta\}$. The {\bf entropy} of $\alpha$ is 
$$H(\alpha):= -\sum_{A\in\alpha} \mu(A) \log\big(\mu(A)\big).$$
We will need a relative version of this quantity as well. So let $\cF \subset \cB$ be a sub-$\sigma$-algebra. Given $A \in \cB$, let $\mu(A|\cF)$ be the conditional expectation of the characteristic function $\chi_A$ of $A$ with respect to $\cF$. The {\bf conditional information function} $I(\alpha|\sF):X \to \R$ is defined by
$$I(\alpha|\sF)(x) := -\log\big(\mu(A_x|\sF)(x)\big)$$
where $A_x$ is the atom of $\alpha$ containing $x$.  The {\bf entropy of $\alpha$ conditioned on $\sF$} is 
$$H(\alpha|\sF) := \int_X I(\alpha | \sF)(x) \, d\mu(x).$$
The {\bf mean entropy} of $\alpha$ given $\sF$ with respect to $T$ is 
$$h(T, \alpha|\sF) := \lim_{n\to\infty} \frac{1}{n+1} H \Big( \bigvee_{i=0}^n   T^{-i}\alpha|\sF\Big).$$
If $\sF$ is $T$-invariant (i.e., if $T^{-1}A \in \sF ~\forall A \in \sF$) then this limit exists. It is well-known that
$$h(T,\alpha|\sF)= H\Big(\alpha | \sF \vee \bigvee_{i=1}^\infty T^{-i}\alpha\Big).$$
Define $h(T|\sF) := \sup_{\alpha} h(T, \alpha|\sF)$ where the supremum is over all partitions $\alpha$ with $H(\alpha|\sF)<\infty$. Let $\tau=\{X,\emptyset\}$ be the minimal $\sigma$-algebra and let $h(T,\alpha)=h(T,\alpha|\tau)$ and $h(T):=h(T|\tau)$. When it is helpful to emphasize the measure we will write $h(T,\mu,\alpha)$ for $h(T,\alpha)$.

We will generalize the next two theorems to actions of free groups.

\begin{thm}[The Abramov-Rohlin Formula]\label{thm:AR}
If $\alpha$ and $\beta$ are any two measurable partitions with $H(\alpha)+H(\beta)<\infty$ then
$$h(T, \alpha \vee \beta) = h(T,\alpha) + h(T, \beta|\alpha^T).$$
\end{thm}
The original Abramov-Rohlin formula, proven in [AR62], was stated in terms of skew-products. The version above is due to Bogensch\"utz and Crauel [BC92]. This formula was generalized in [WZ92] to amenable group actions. See [Da01] for an alternative proof using orbit equivalence theory.


\begin{thm}[Yuzvinskii's Addition Formula]\label{thm:yuz}
Let $\sG$ be a separable compact group, $T:\sG \to \sG$ a Haar measure-preserving homomorphism and $\cN<\sG$ a closed normal $T$-invariant subgroup. Let $T_{\sG/\cN}:\sG/\cN \to \sG/\cN$ be the induced homomorphism and $T_{\cN}:\cN \to \cN$ the restriction of $T$ to $\cN$. Then
$$h(T)=h(T_{\sG/\cN}) + h(T_\cN)$$
where each entropy rate is with respect to the Haar probability measure on $\sG, \sG/\cN$ and $\cN$ respectively.
\end{thm}
This was proven first in [Yu65]. R. K. Thomas [Th71] enhanced this formula to skew-product actions. In [LSW90] it was generalized to actions of $\Z^d$. There are related results in [LSW90, De06, DS07, BM08].  In a very recent preprint [LS09], Lind and Schmidt have extended Yuzvinksii's formula to all algebraic actions of an arbitrary amenable group.

\subsection{Free groups and semigroups}\label{sub:f}
From now on, let $G=\langle s_1,\ldots, s_r\rangle$ denote either a free group or a free semigroup with identity. If $G$ is a group, let $S=\{s_1^{\pm 1},\ldots, s_r^{\pm 1}\}$. In the semigroup case, let $S=\{s_1,\ldots,s_r\}$. Let $|\cdot|:G \to \R$ denote the word metric with respect to $S$.

We will write $G \c^T (X,\cB,\mu)$ to denote that $T:G \to \ed(X,\cB,\mu)$ is a homomorphism from $G$ into the semigroup of measure-preserving transformations of $(X,\cB,\mu)$ which we will always assume is a standard probability space. Measure-preserving means that for all $g\in G$ and $E \in \cB$, $\mu(T_g^{-1}E)=\mu(E)$.  If $\alpha$ is a partition of $X$ and $Q \subset G$ is finite, then $\alpha^Q:= \bigvee_{q\in Q} T_q^{-1} \alpha$. To simplify notation, let $\alpha^n := \alpha^{B(e,n)}$ where $B(e,n) \subset G$ is the ball of radius $n$ centered at the identity element with respect to the word metric. Define
\begin{eqnarray*}
F(T,\alpha)&:=&(1-2r)H(\alpha) + \sum_{i=1}^r H(\alpha \vee T_{s_i}^{-1}\alpha)\\
f(T,\alpha)&:=&\inf_{n>0} F(T,\alpha^n).
\end{eqnarray*}
In [Bo08a], it is proven that if $\alpha$ generates (i.e., the smallest $G$-invariant $\sigma$-algebra containing $\alpha$, denoted $\alpha^G$, equals $\cB$ up to sets of measure zero) and if $\beta$ also generates and $H(\alpha)+H(\beta)<\infty$ then $f(T,\alpha)=f(T,\beta)$. This common number is called the $f$-invariant of the action (denoted $f(T)$). It is a measure-conjugacy invariant. It is our substitute for entropy rate. Unlike the classical case, $f(T)$ is well-defined only if there exists a generating partition $\alpha$ with $H(\alpha)<\infty$. Also $f(T)$ can take negative values.

We will need the following relative versions. If $\cF \subset \cB$ is a sub-$\sigma$-algebra then define $\cF^Q$ and $\cF^n$ similarly to the above and let
\begin{eqnarray*}
F(T,\alpha|\cF)&:=&(1-2r)H(\alpha|\cF) + \sum_{i=1}^r H(\alpha \vee T_{s_i}^{-1}\alpha|\cF \vee T_{s_i}^{-1}\cF)\\
f(T,\alpha|\cF)&:=& \inf_{n>0} F(T,\alpha^n|\cF^n).
\end{eqnarray*}
When $T$ is fixed we will write $f(\alpha|\sF)$ instead of $f(T,\alpha|\sF)$. When it is helpful to emphasize the measure we will write $f(T,\mu,\alpha|\sF)$ instead of $f(T,\alpha|\sF)$.

The next theorem generalizes Abramov-Rohlin's formula.
\begin{thm}\label{thm:freeAR}\label{thm:abramov}
Let $G \c^T (X,\cB,\mu)$. If $\alpha$ and $\beta$ are partitions of $X$ with $H(\alpha) + H(\beta) <\infty$ then
$$f(T,\alpha\vee \beta) = f(T,\alpha) + f(T,\beta|\alpha^G).$$
\end{thm}

To illustrate, a simple calculation shows that if $X$ that has exactly $n$ elements and $\mu$ is the uniform probability measure on $X$ then $f(T) = (1-r)\log(n)$. Note this is negative if $n>1$ and $r>1$. The above theorem and standard skew-product arguments now imply:
\begin{cor}\label{cor:finiteto1}
Let $G \c^T (X,\cB,\mu)$ be an ergodic $G$-system. Let $G \c^S (Y,\cC,\nu)$ and suppose there is a $n$-to-1 factor map $\phi:X \to Y$ (i.e., $\phi_*\mu=\nu$, $\phi(T_gx)=S_g\phi(x)$ for a.e. $x$ and $|\phi^{-1}(y)|=n$ for a.e. $y$). Then
$$f(S) = (r-1)\log(n) + f(T)$$
whenever $f(S)$ and $f(T)$ are well-defined.
\end{cor}

The next result generalizes Yuzvinskii's addition formula.

\begin{thm}\label{thm:freeyuz}
Let $G=\langle s_1,\ldots,s_r\rangle$ be a rank $r$ free group or semigroup. Let $\sG$ be a separable compact group which is either totally disconnected, a Lie group, or a finite-dimensional connected abelian group. Let $T_{\sG}:G \to \ed(\sG)$ be a homomorphism and let $\cN<\sG$ be a closed normal $G$-invariant subgroup. Let $T_{\cN}:G \to \ed(\cN)$ and $T_{\sG/\cN}:G \to \ed(\sG/\cN)$ be the induced homomorphisms. Then
$$f(T_{\sG})=f(T_{\sG/\cN}) + f(T_{\cN})$$
whenever $f(T_\sG), f(T_{\sG/\cN})$ and $f(T_{\cN})$ are well-defined. The numbers $f(T_\sG), f(T_{\sG/\cN})$ and $f(T_{\cN})$ are computed with respect to Haar probability measure on $\sG, \sG/\cN$ and $\cN$ respectively.
\end{thm}

I conjecture that the above result holds for all separable compact groups $\sG$. In [El99] it was proven that there is no invariant for nonabelian free group actions (and many other nonamenable groups) that satisfies a Yuzvinskii-type formula under some rather general assumptions on the invariant. But the $f$-invariant does not satisfy these because it can take negative values.

To illustrate, let us recall the following example from [OW87]. Let $G=\langle s_1,s_2\rangle$ be the rank 2 free group. Let $\sG=(\Z/2\Z)^G$ be the set of all functions from $G \to \Z/2\Z$. It is a group under pointwise addition. It can be considered as the product of $G$ copies of $\Z/2\Z$. By Tychonoff's theorem, it is compact. Let $\cN < \cG$ the subgroup of constant functions. So $|\cN|=2$.  For $g \in G$, define $T_g:\sG \to \sG$ by $T_g x(f)=x(g^{-1}f)$. This action preserves Haar measure on $\sG$ and leaves $\cN$ invariant. 

In [OW87], it is pointed out that $\sG/\cN$ is isomorphic to $\sG \times \sG \cong (\Z/2\Z \times \Z/2\Z)^G$. Indeed the factor map $\phi: (\Z/2\Z)^G \to (\Z/2\Z \times \Z/2\Z)^G$ defined by
$$\phi(x)(g) = \big( x(g) + x(s_1g), x(g) + x(s_2g) \big)$$
defines an isomorphism $\sG/\cN \cong (\Z/2\Z \times \Z/2\Z)^G \cong \sG \times \sG$. So, the above theorem implies that
$$f(T_{\sG}) = f(T_{\cN}) + f( T_{\sG \times \sG}).$$
This is easily verified. $T_{\sG}$ and $T_{\sG \times \sG}$ are both Bernoulli shift actions. From one of the main results of [Bo08a], it follows that $f(T_\sG)=\log(2)$ and $f(T_{\sG \times \sG})=\log(4)$. The action of $G$ on $\cN$ is trivial and it is easy to verify that $f(T_{\cN})=-\log(2)$ as required. Alternatively, since the above factor map is 2-1, this formula can be derived from corollary \ref{cor:finiteto1}.

\subsection{An alternative formulation of the $f$-invariant}

We will prove the following formula for the $f$-invariant that helps enable the transfer of results from the classical case to the case of free groups. To explain, let $G \c^T (X,\cB,\mu)$, $\sF\subset \cB$ be a $T$-invariant sub-$\sigma$-algebra and $\alpha$ be a partition of $X$. Define
\begin{eqnarray*}
F_*(T,\alpha|\sF)&:=& (1-r)H(\alpha|\sF) + \sum_{i=1}^r h(T_{s_i}, \alpha|\sF).\\
f_*(T,\alpha|\sF) &:=& \inf_{n>0} F_*(T,\alpha^n|\sF).
\end{eqnarray*}
In \S \ref{sec:f=f_*} we prove that $f_*(T,\alpha|\sF)=f(T,\alpha|\sF)$. 

\subsection{Organization}

\S \ref{sec:notation} explains notation used throughout the paper. \S \ref{sec:standard} is a review of classical entropy theory. \S \ref{sec:spaceofpartitions} is a study of the space of partitions of $X$. \S \ref{sec:alternative} introduces $f_*$ and proves that $f$ and $f_*$ are measure-conjugacy invariants (using the main theorem of \S \ref{sec:spaceofpartitions}). \S \ref{sec:markovprocesses} introduces Markov processes and proves that $F(\alpha)=f(\alpha)$ for such processes. \S \ref{sec:markovchains} develops a constructive approach to Markov processes via transition matrices and symbolic dynamics. \S \ref{sec:examples} presents three examples of Markov processes. \S \ref{sec:f=f_*} proves that $f=f_*$ using Markov approximations to an arbitrary system. This is then used to give a short proof of theorem \ref{thm:freeAR} in \S \ref{sec:abramov}. \S \ref{sec:characterization} proves that if a process $(T,X,\mu,\alpha)$ satisfies $F(\alpha)=f(\alpha)$ then it must be Markov. \S \ref{sec:limits} proves more approximation results that are used in \S \ref{sec:yuz} to prove theorem \ref{thm:freeyuz}.

{\bf Acknowledgements.}
I would like to thank Russ Lyons for suggesting that I think about the isomorphism problem for Bernoulli shifts over a nonabelian free group and for many useful conversations along the way. I'd also like to thank Benjy Weiss for asking whether the infinite entropy Bernoulli shift over a nonabelian free group could be finitely generated. That question is answered in [Bo08b] and a different proof is provided in \S \ref{sec:limits}.

\section{Notation}\label{sec:notation}
In general, $G:=\langle s_1,\ldots, s_r\rangle$ denotes either a free group or free semigroup with $1$.  If $G$ is a group, let $S=\{s_1^{\pm 1},\ldots, s_r^{\pm 1}\}$. In the semigroup case, let $S=\{s_1,\ldots,s_r\}$. 

We will write $G \c^T (X,\cB,\mu)$ to denote that $T:G \to \ed(X,\cB,\mu)$ is a homomorphism from $G$ into the semigroup of measure-preserving transformations of $(X,\cB,\mu)$ which we will always assume is a standard probability space. Measure-preserving means that for all $g\in G$ and $E \in \cB$, $\mu(T_g^{-1}E)=\mu(E)$. When convenient we will ignore the $\sigma$-algebra by writing $G \c^T (X,\mu)$ instead. The triple $(T,X,\mu)$ is a called a {\bf $G$-system} or an {\bf action of $G$}.

We use $\alpha,\beta$ to denote partitions of $X$ into at most countably many measurable subsets.

\section{Review of classical entropy theory}\label{sec:standard}

Fix a probability space $(X,\cB,\mu)$.
\begin{defn}
A {\bf partition} $\alpha=\{A_1, A_2,\ldots\}$ is a pairwise disjoint collection of measurable subsets $A_i$ of $X$ such that $\cup_i A_i = X$. The sets $A_i$ are called the {\bf partition elements} of $\alpha$. Alternatively, they are called the {\bf atoms} of $\alpha$.  Unless stated otherwise, all partitions in this paper are either finite or countable infinite.


\end{defn}

\begin{defn}
If $\alpha$ and $\beta$ are partitions of $X$  then the {\bf join}  of $\alpha$ and $\beta$ is the common refinement partition $\alpha \vee \beta=\{A \cap B\,|\,A \in \alpha, ~B\in \beta\} $. By abuse of notation, we will sometimes identify a join with the $\sigma$-algebra that it generates. Thus if $\alpha_1,\alpha_2,\ldots$ is a sequence of partitions then $\bigvee_{i=1}^\infty \alpha_i$ is identified with the smallest $\sigma$-algebra of $X$ that contains every atom of $\alpha_i$ for all $i$.
\end{defn}

\begin{defn}
The {\bf information function} $I(\alpha):X \to \R$ corresponding to a partition $\alpha$ is defined by
$$I(\alpha)(x) = -\log(\mu(A_x))$$
where $A_x$ is the atom of $\alpha$ containing $x$.
\end{defn}

\begin{defn}
The {\bf entropy} $H(\alpha)$ of $\alpha$ is defined by
$$H(\alpha) = - \sum_{A\in \alpha} \mu(A)\log(\mu(A)) = \int_{x\in X} I(\alpha)(x) \, d\mu(x). $$
By convention $0\log(0)=0$.
\end{defn}

\begin{defn}
Let $G$ be a group (or semigroup with $1$) acting on $(X,\cB,\mu)$. Let $\alpha$ be a partition. Let $\alpha^G$ be the smallest $G$-invariant $\sigma$-algebra containing the atoms of $\alpha$. Then $\alpha$ is {\bf generating} (with respect to the given action of $G$) if for every measurable set $A \subset X$ there exists a set $A' \in \alpha^G$ such that $\mu(A \Delta A')=0$.
\end{defn}

\begin{defn}
Let $T:X \to X$ be a measure-preserving transformation. The {\bf mean entropy} of a partition $\alpha$ of $X$ is
$$h(T,\alpha):=\lim_{n \to \infty} \frac{H( \bigvee_{i=0}^n T^{-i} \alpha)}{n+1} = \lim_{n\to\infty}  H(T^{-n-1}\alpha | \bigvee_{i=0}^n T^{-i}\alpha).$$
A. N. Kolmogorov proved [Ko58, Ko59] that if $\alpha$ and $\beta$ are finite-entropy generating partitions then $h(T,\alpha)=h(T,\beta)$. Y. Sinai proved [Si59] that if $\alpha$ is any finite-entropy partition and $\beta$ is generating partition then $h(T,\alpha) \le h(T,\beta)$. So the {\bf entropy} of the system is defined by $h(T):=\sup_{\alpha} h(T,\alpha)$ where the sup is over all finite-entropy partitions $\alpha$. This defines an isomorphism invariant of the system $(T,X,\mu)$.
\end{defn}

\begin{defn}\label{defn:conditional}
Let $\sF$ be a $\sigma$-algebra contained in the $\sigma$-algebra of all measurable subsets of $X$. Given a partition $\alpha$, define the {\bf conditional information function} $I(\alpha|\sF):X \to \R$ by
$$I(\alpha|\sF)(x) = -\log\big(\mu(A_x|\sF)(x)\big)$$
where $A_x$ is the atom of $\alpha$ containing $x$. Here, if $A \subset X$ is measurable then $\mu(A|\sF):X \to \R$ is the conditional expectation of $\chi_{A}$, the characteristic function of $A$, with respect to the $\sigma$-algebra $\sF$. The {\bf conditional entropy of $\alpha$ with respect to $\sF$} is defined by
$$H(\alpha|\sF) = \int_X I(\alpha | \sF)(x) \, d\mu(x).$$

If $\beta$ is a partition then, by abuse of notation, we can identify $\beta$ with the $\sigma$-algebra equal to the set of all unions of partition elements of $\beta$. Through this identification, $I(\alpha|\beta)$ and $H(\alpha|\beta)$ are well-defined.
\end{defn}

\begin{defn}
Let $T:X \to X$ be a measure-preserving transformation. If $\sF\subset \cB$ is a $T$-invariant sub-$\sigma$-algebra then the entropy rate of $\alpha$ conditioned on $\sF$ is
$$h(T,\alpha|\sF):=\lim_{n\to\infty} \frac{1}{n+1} H\Big( \bigvee_{i=0}^n T^{-i} \alpha|\sF\Big)=\lim_{n\to\infty} H\Big(T_s^{-n-1}\alpha |\sF \vee \bigvee_{i=0}^n T_s^{-i}\alpha\Big).$$
\end{defn}

\begin{lem}\label{lem:relative}
For any two partitions $\alpha, \beta$ and for any two $\sigma$-algebras $\sF_1, \sF_2$ with $\sF_1 \subset \sF_2$,
$$H(\alpha \vee \beta) = H(\alpha) + H(\beta|\alpha),$$
$$H(\alpha  | \sF_2) \le H(\alpha  | \sF_1)$$
with equality if and only if $\mu(A |\sF_2) = \mu(A |\sF_1)$ a.e. for every $A \in \alpha$. In particular $H(\alpha|\beta) \le H(\alpha)$ and equality occurs iff $\alpha$ and $\beta$ are {\bf independent} (i.e., $\forall A \in \alpha, B \in \beta, \mu(A\cap B)=\mu(A)\mu(B)$).
\end{lem}
\begin{proof}
This is well-known. For example, see [Gl03, Proposition 14.16, page 255].
\end{proof}



\section{The space of partitions}\label{sec:spaceofpartitions}
Let $G \c^T (X,\cB,\mu)$. Let $\cP$ be the set of all partitions $\alpha$ of $X$ such that $H(\alpha)<\infty$. We identify partitions if they agree up to measure zero. The main theorem below is needed to prove that the $f$-invariant is a measure-conjugacy invariant (which is concluded in \S \ref{sec:alternative}). The splittings concept introduced below will be useful in our study of Markov processes.

\begin{defn}[Rohlin distance]\label{defn:rohlin}
Define $d:\sP \times \sP \to \R$ by
$$d(\alpha,\beta) = H(\alpha|\beta) + H(\beta|\alpha) = 2H(\alpha \vee \beta) - H(\alpha) - H(\beta).$$
By [Pa69, theorem 5.22, page 62] this defines a distance function on $\cP$. The action of $G$ on $\sP$ is isometric. I.e., if $g \in G$, $\alpha, \beta \in \sP$ then $d(T_g^{-1}\alpha, T_g^{-1}\beta) = d(\alpha,\beta)$.
\end{defn}

\begin{defn}
Let $\alpha$ and $\beta$ be partitions. If, for every atom $A \in \alpha$ there exists an atom $B \in \beta$ such that $\mu(A-B)=0$ (i.e., $A\subset B$ up to a measure zero set) then we say $\alpha$ {\bf refines} $\beta$. Equivalently, $\beta$ is a {\bf coarsening} of $\alpha$. This is denoted by $\beta \le \alpha$.
\end{defn}

\begin{defn}
If $\alpha$ is a partition of $X$ and $Q \subset G$ is finite then let $\alpha^Q = \bigvee_{q\in Q} T_q^{-1}\alpha$. Two partitions $\alpha, \beta \in \cP$ are {\bf equivalent} if there exists finite sets $Q,P \subset G$ such that $\alpha \le \beta^P$ and $\beta \le \alpha^Q$.
\end{defn}

\begin{thm}\label{thm:dense}
If $\alpha, \beta \in \sP$ are generating partitions and $\epsilon>0$ then there exists a $\gamma \in \cP$ that is equivalent to $\alpha$ such that $d(\gamma,\beta)<\epsilon$. In other words, the equivalence class of $\alpha$ is dense in the space of all generating partitions. 
\end{thm}
For a proof, we refer the reader to [Bo08b]. The notation there differs from the notation here in one respect: $\alpha^Q$ is defined to be $\bigvee_{q\in Q} T_q \alpha$. Also, only groups, rather than semigroups are treated in [Bo08b]. However the proof requires only minor obvious changes to extend it to the semigroup case. 

\subsection{Splittings}
Let us assume now (and for the rest of the paper) that $G=\langle s_1,\ldots,s_r\rangle$ is a free group or free semigroup with $1$. If $G$ is a group then let $S=\{s_1,\ldots, s_r, s_1^{-1},\ldots,s_r^{-1}\}$. If $G$ is only a semigroup, let $S=\{s_1,\ldots,s_r\}$. Let $G \c^T (X,\cB,\mu)$


\begin{defn}\label{defn:splitting}
Let $\alpha$ be a partition. A {\bf simple splitting} of $\alpha$ is a partition $\sigma$ of the form $\sigma=\alpha \vee T_s^{-1}\beta$ where $s\in S$ and $\beta$ is a coarsening of $\alpha$.

A {\bf splitting} of $\alpha$ is any partition $\sigma$ that can be obtained from $\alpha$ by a sequence of simple splittings. In other words, there exist partitions $\alpha_0, \alpha_1,\ldots,\alpha_m$ such that $\alpha_0=\alpha$, $\alpha_m = \sigma$ and $\alpha_{i+1}$ is a simple splitting of $\alpha_i$ for all $1 \le i < m$.
\end{defn}

\begin{remark}
In [Bo08b], an $S$-splitting of $\alpha$ is defined to be a partition $\sigma$ of the form $\sigma=\alpha \vee T_s \beta$ for $s\in S$. The definition given above is necessary to accommodate the case when $G$ is merely a semigroup.
\end{remark}

\begin{defn}\label{defn:cayley}
The {\bf right-Cayley graph} $\Gamma$ of $(G,S)$ is defined as follows. The vertex set of $\Gamma$ is $G$. For every $s \in S$ and every $g \in G$ there is a directed edge from $g$ to $gs$ labeled $s$. There are no other edges. 

The {\bf induced right-subgraph} of a subset $F \subset G$ is the largest subgraph of $\Gamma$ with vertex set $F$. A subset $F \subset G$ is {\bf right-connected} if its induced right-subgraph in $\Gamma$ is connected.
\end{defn}

\begin{lem}\label{lem:splittings}
If $\alpha, \beta \in \sP$, $\alpha$ refines $\beta$ and $F \subset G$ is finite, right-connected and contains the identity element $e$ then
$$\alpha \vee \bigvee_{f \in F} T_f^{-1}\beta$$
is a splitting of $\alpha$.
\end{lem}

\begin{proof}
We prove this by induction on $|F|$. If $|F|=1$ then $F=\{e\}$ and the statement is trivial. Let $f_0 \in F-\{e\}$ be such that $F_1=F-\{f_0\}$ is right-connected. To see that such an $f_0$ exists, choose a spanning tree for the induced right-subgraph of $F$. Let $f_0$ be any leaf of this tree that is not equal to $e$.

By induction, $\alpha_1 := \alpha \vee \bigvee_{f \in F_1} T_f^{-1}\beta $ is a splitting of $\alpha$. Since $F$ is right-connected, there exists an element $f_1 \in F_1$ and an element $s_1 \in S$ such that $f_1s_1=f_0$. Since $f_1 \in F_1$, $\alpha_1$ refines $T_{f_1}^{-1}\beta$. Thus
 $$\alpha \vee \bigvee_{f \in F} T_f^{-1}\beta = \alpha_1 \vee T_{f_0}^{-1}\beta= \alpha_1 \vee T_{s_1}^{-1}(T_{f_1}^{-1}\beta)$$ is a splitting of $\alpha$.
\end{proof}

To ease notation, let $\alpha^n = \alpha^{B(e,n)}$ where $B(e,n)$ denotes the ball of radius $n$ centered at the identity element in $G$ with respect to the word metric induced by $S$.

\begin{prop}\label{prop:splitting}
Let $\alpha,\beta \in \sP$. Suppose there are $n,m \in \N$ such that $\alpha \le \beta^n \le \alpha^m$. Then $\alpha^m$ is a splitting of $\beta$.
\end{prop}

\begin{proof}
By the previous lemma, $\beta^n \vee \alpha^m=\alpha^m$ is a splitting of $\beta$.
\end{proof}

\section{An alternative formula for the $f$-invariant}\label{sec:alternative}


Recall the definitions of $F$ and $F_*$ from the introduction. We will write $F(\alpha|\sF)$ for $F(T,\alpha|\sF)$ when $T$ is clear. Similar statements apply to $f(\alpha|\sF)$, $F_*(\alpha|\sF)$, etc.

\begin{prop}\label{prop:splittingmonotone}
Let $G\c^T (X,\cB,\mu)$. If $\sF \subset \cB$ is any $T(G)$-invariant $\sigma$-algebra, $\alpha$ is any partition with $H(\alpha)<\infty$, and $\sigma$ is any splitting (definition \ref{defn:splitting}) of $\alpha$ then $F(\sigma|\sF) \le F(\alpha|\sF)$ and $F_*(\sigma|\sF) \le F_*(\sigma|\sF)$.
\end{prop}
\begin{proof}
It suffices to consider the case in which $\sigma$ is a simple splitting. So, there exists $t\in S$ and a coarsening $\beta $ of $\alpha$ such that $\sigma = \alpha \vee T_t^{-1}\beta$. We will assume that $t\in \{s_1,\ldots,s_r\}$. The proof in the case that $t\in \{s_1^{-1},\ldots,s_r^{-1}\}$ is similar. Using lemma \ref{lem:relative}, it follows that
\begin{eqnarray*}
F(\sigma|\sF) &=& F(\alpha|\sF) + (1-2r)H(\sigma|\alpha \vee \sF) + \sum_{i=1}^r H(\sigma \vee T_{s_i}^{-1}\sigma|\alpha \vee T_{s_i}^{-1}\alpha \vee \sF).
\end{eqnarray*}
Note that
\begin{eqnarray*}
 H(\sigma \vee T^{-1}_s\sigma|\alpha \vee T_s^{-1}\alpha \vee \sF) &\le &  H(\sigma |\alpha \vee T_s^{-1}\alpha \vee \sF)
+  H( T_s^{-1}\sigma|\alpha \vee T_s^{-1}\alpha \vee \sF).
\end{eqnarray*}
Since $\sigma$ is refined by $\alpha \vee T_t^{-1}\alpha$ it follows that $H(\sigma |\alpha \vee T_t^{-1}\alpha \vee \sF) = 0.$ Without loss of generality, $t=s_r$. Thus,
\begin{eqnarray*}
F(\sigma|\sF) &\le& F(\alpha|\sF) + \Big( \sum_{i=1}^{r-1} H(\sigma |\alpha \vee T_{s_i}^{-1}\alpha \vee \sF)
+  H( T_{s_i}^{-1}\sigma|\alpha \vee T_{s_i}^{-1}\alpha \vee \sF) - 2H(\sigma|\alpha \vee \sF)\Big)\\
&&+ H( T_{s_r}^{-1}\sigma|\alpha \vee T_s^{-1}\alpha \vee \sF) - H(\sigma|\alpha \vee \sF).
\end{eqnarray*}
Since $H(\sigma |\alpha \vee T_{s_i}^{-1}\alpha \vee \sF) \le H(\sigma|\alpha\vee\sF)$ and 
$$ H( T_{s_i}^{-1}\sigma|\alpha \vee T_{s_i}^{-1}\alpha \vee \sF) \le H( T_{s_i}^{-1}\sigma| T_{s_i}^{-1}\alpha \vee \sF) =H(\sigma|\alpha\vee\sF)$$
it follows that $F(\sigma|\sF) \le F(\alpha|\sF)$ as claimed.

The proof in the case of $F_*$ is similar. By a well-known relative version of theorem \ref{thm:AR}, if $s\in S$ then
$$h(T_s, \sigma | \sF)= h(T_s,\alpha|\sF) + h(T_s,\sigma | \alpha^s \vee \sF)$$
where $\alpha^s$ is the smallest $T_s$-invariant $\sigma$-algebra containing $\alpha$.  As above, assume $t=s_r$. Thus,
\begin{eqnarray*}
F_*(\sigma|\sF) &=& F_*(\alpha|\sF) + (1-r)H(\sigma|\alpha\vee\sF) + \sum_{i=1}^r h(T_{s_i},\sigma|\alpha^{s_i} \vee \sF) \\
&=& F_*(\alpha|\sF) + (1-r)H(\sigma|\alpha\vee\sF) + \sum_{i=1}^{r-1} h(T_{s_i},\sigma|\alpha^{s_i} \vee \sF)\\
&=& F_*(\alpha|\sF) + \sum_{i=1}^{r-1} \big(h(T_{s_i},\sigma|\alpha^{s_i} \vee \sF)-H(\sigma|\alpha\vee\sF) \big).
\end{eqnarray*}
The second equality occurs because $\alpha \vee T_{s_r}^{-1}\alpha$ refines $\sigma$ implies $h(T_{s_r}, \sigma|\alpha^{s_r})=0.$ Since $h(T_s,\sigma|\alpha^s \vee \sF) \le H(\sigma|\alpha \vee \sF)$ for each $s\in S$, the above equality implies the lemma.
\end{proof}

\begin{defn}\label{defn:f}
If $\sF\subset \cB$ is any $T(G)$-invariant $\sigma$-algebra then define
$$f(\alpha|\sF) := \lim_{n\to \infty} F(\alpha^n |\sF) = \inf_n F(\alpha^n|\sF),$$
$$f_*(\alpha|\sF):=\lim_{n\to \infty} F_*(\alpha^n |\sF) = \inf_n F_*(\alpha^n|\sF).$$
The previous proposition and proposition \ref{prop:splitting} implies that this is well-defined. When we need to emphasize the dependence on $\mu$ and/or $T$ we will write $f(\mu,\alpha|\sF)$ or $f(T,\alpha|\sF)$ for $f(\alpha|\sF)$ and similarly for $F, F_*,f_*$.
\end{defn}

Next we investigate the continuity properties of these functions.
\begin{prop}\label{prop:partitioncontinuity}
Let $G \c^T (X,\cB,\mu)$. Let $\sP$ be the space of partitions $\alpha$ of $X$ with $H(\alpha)<\infty$. Endow $\sP$ with the topology induced by the Rohlin distance (definition \ref{defn:rohlin}). Then $F$ and $F_*$ are continuous on $\sP$ and $f$ and $f_*$ are upper semi-continuous on $\sP$.
\end{prop}
\begin{proof}
It is immediate that $F$ is continuous. Y. Sinai proved that for every $s\in S$, the function $\alpha \mapsto h(T_s, \alpha)$ is continuous on $\sP$ (see for example [Gl03]). From this it follows that $F_*$ is continuous. The function $\alpha \mapsto \alpha^n$ is continuous by lemma 4.2 of [Bo08b] (it is also an easy exercise). Thus, each of $f$ and $f_*$ is an infimum of a sequence of continuous functions. This implies that $f$ and $f_*$ are upper semi-continuous.
\end{proof}

Later (in lemma \ref{lem:mucontinuous}) we investigate the continuity properties of the above functions in the variable $\mu$ rather than $\alpha$.
\begin{thm}
Let $G\c^T (X,\cB,\mu)$. If $\alpha$ and $\beta$ are any two generating partitions with $H(\alpha)+H(\beta)<\infty$ and $\sF\subset \cB$ is any $T(G)$-invariant $\sigma$-algebra then $f(\alpha|\sF)=f(\beta|\sF)$ and $f_*(\alpha|\sF)=f_*(\beta|\sF)$. Thus, we can define $f(T|\sF)=f(\alpha|\sF)$ and $f_*(T|\sF)=f_*(\alpha|\sF)$ for any finite-entropy generating partition $\alpha$.
\end{thm}
\begin{proof}
This follows from theorem \ref{thm:dense} and propositions \ref{prop:splitting}, \ref{prop:splittingmonotone} and \ref{prop:partitioncontinuity}. To see this, note that from theorem \ref{thm:dense}, there exists a sequence of partitions $\alpha_n$ with $d(\alpha_n,\beta) \to 0$ as $n\to\infty$ and integers $m(n), p(n)$ with $\alpha_n \le \alpha^{m(n)} \le \alpha_n^{p(n)}$. Thus $\alpha \le \alpha_n^{p(n)} \le \alpha^{m(n)+p(n)}$. From proposition \ref{prop:splitting}, this implies that $\alpha^{m(n)+p(n)}$ is a splitting of $\alpha_n$. Thus for every $k$, $\alpha^{m(n)+p(n)+k}$ is a splitting of $\alpha_n^k$. By proposition \ref{prop:splittingmonotone}, this implies $F(\alpha^{m(n)+p(n)+k}) \le F(\alpha^k_n)$. The definition of $f$ now implies that $f(\alpha) \le f(\alpha_n)$ for all $n$. By the previous proposition, $f$ is upper semi-continuous. Since $\alpha_n$ converges to $\beta$, this implies $f(\alpha) \le f(\beta)$. By reversing the roles of $\alpha$ and $\beta$ we obtain the reverse inequality. Hence $f(\alpha)=f(\beta)$ as claimed. The conditional case and the case of $f_*$ in place of $f$ are similar.

\end{proof}

In section \ref{sec:f=f_*}, it is proven that $f(T|\sF)=f_*(T|\sF)$. The proof uses Markov processes which are studied next.

\section{Markov Processes}\label{sec:markovprocesses}


\begin{defn}
A {\bf$G$-process} is a quadruple $(T,X,\mu,\alpha)$ where $G \c^T (X,\cB,\mu)$ and $\alpha$ is a partition of $X$.
\end{defn}

\begin{defn}
Two processes $(T,X,\mu,\alpha)$ and $(U,Y,\nu,\beta)$ are {\bf isomorphic} if there exists conull sets $X' \subset X, Y' \subset Y$ and a measurable map $\phi:X' \to Y'$ with measurable inverse $\phi^{-1}:Y' \to X'$ such that $\phi_*\mu=\nu$, $\phi(T_gx)=U_g\phi(x) \forall g\in G, x\in X'$ and $\phi_*\alpha=\beta$ (i.e., $\phi$ induces a bijection from $\alpha$ to $\beta$).
\end{defn}

\begin{defn}\label{defn:leftcayley}
The {\bf left-Cayley graph} $\Gamma_L$ of $(G,S)$ is defined as follows. Its vertex set is $G$ and for every $g\in G$ and $s\in S$ there is a directed edge from $g$ to $sg$ there are no other edges. If $F \subset G$ then the {\bf left-subgraph induced by $F$} is the subgraph of $\Gamma_L$ that has vertex set equal to $F$ and contains every edge in $\Gamma_L$ whose endpoints are in $F$. A set $F$ is {\bf left-connected} if the left-subgraph induced by $F$ is connected.

This is opposite the right-Cayley graph which was defined earlier (definition \ref{defn:cayley}).
\end{defn}

\begin{defn}
For all $g_1, g_2 \in G$ let $\past(g_1;g_2) \subset G$ be the set of all $f\in G$ such that every path in the left-Cayley graph $\Gamma_L$ from $f$ to $g_1$ passes through $g_2$.
\end{defn}

\begin{defn}
For any measure $\mu$ on $X$, any Borel set $A \subset X$ and any $\sigma$-algebra $\sF$, let $\mu(A|\sF):X \to \R$ denote the conditional expectation of the characteristic function $\chi_A$ of $A$ with respect to $\sF$.
\end{defn}

\begin{defn}
A process $(T,X,\mu,\alpha)$ is a {\bf Markov process} if for every $s\in S$, $g\in G$  and every $A \in \alpha$
$$\mu\Big(T_{sg}^{-1}A \big| \bigvee_{f \in \past(sg;g)} T_f^{-1}\alpha\Big)(x) =  \mu(T_{sg}^{-1}A | T_g^{-1}\alpha)(x)=\mu(T_s^{-1}A | \alpha)(x)$$
for $\mu$-a.e. $x\in X$. The second equality above is automatically true since $T_g$ preserves $\mu$. By lemma \ref{lem:relative} this is equivalent to:
$$H\Big(T_{sg}^{-1}\alpha \big| \bigvee_{f \in \past(sg;g)} T_f^{-1}\alpha\Big) = H(T_{sg}^{-1}\alpha | T_g^{-1}\alpha) = H(T_s^{-1}\alpha | \alpha).$$
for every $g\in G$.
\end{defn}

The main result of this section is:
\begin{thm}\label{thm:f=F}
If $(T,X,\mu,\alpha)$ is a Markov process and $\beta$ is a coarsening of $\alpha$ then $f(T,\alpha|\beta^G)=F(T,\alpha|\beta^G)=F_*(T,\alpha|\beta^G)=f_*(T,\alpha|\beta^G)$ where $\beta^G$ is the smallest $T(G)$-invariant $\sigma$-algebra containing $\beta$.
\end{thm}
In section \ref{sec:characterization}, we prove a converse: if $F(T,\alpha)=f(T,\alpha)$ then $(T,X,\mu,\alpha)$ is Markov. In order to prove the above, we will need some lemmas.

\begin{lem}\label{lem:additive}
If $\beta \le \alpha$ are partitions, $\sF_1 \subset \sF_2$ are $\sigma$-algebras and $H(\alpha|\sF_1)=H(\alpha|\sF_2)$, then $H(\beta|\sF_1)=H(\beta|\sF_2)$.
\end{lem}
\begin{proof}
This follows from the fact that conditional expectation is additive.
\end{proof}

\begin{lem}\label{lem:splitting1}
If $(T, X,\mu,\alpha)$ is a Markov process and $\sigma$ is a splitting of $\alpha$ then $(T, X,\mu,\sigma)$ is a Markov process.
\end{lem}

\begin{proof}
By induction, we may assume that $\sigma$ is a simple splitting of $\alpha$. So there exists a $t\in S$ and a coarsening $\beta$ of $\alpha$ such that $\sigma=\alpha \vee T_t^{-1}\beta$. It suffices to prove that
$$H\Big(T_{sg}^{-1}\sigma ~\big|~ \bigvee_{f \in \past(sg;g)} T_f^{-1}\sigma\Big) = H(T_{sg}^{-1}\sigma ~|~ T_g^{-1}\sigma)$$
for every $s\in S$ and $g\in G$.

{\bf Case 1}. Assume $s\ne t$. Then $f \in \past(sg;g)$ implies $tf \in \past(sg;g)$. So, $\bigvee_{f \in \past(sg;g)} T_f^{-1}\sigma = \bigvee_{f \in \past(sg;g)} T_f^{-1}\alpha$. Thus,
\begin{eqnarray}
H\Big(T_{sg}^{-1}\sigma \big| \bigvee_{f \in \past(sg;g)} T_f^{-1}\sigma\Big) &=& H\Big(T_{sg}^{-1}\sigma\big|\bigvee_{f \in \past(sg;g)} T_f^{-1}\alpha\Big)\\
&=&  H\Big(T_{sg}^{-1}\alpha \big|\bigvee_{f \in \past(sg;g)} T_f^{-1}\alpha\Big)\label{eqn:ms1}\\
&& + H\Big(T_{sg}^{-1}T_t^{-1}\beta \big| T_{sg}^{-1}\alpha \vee \bigvee_{f \in \past(sg;g)} T_f^{-1}\alpha\Big).\label{eqn:ms1.5}
\end{eqnarray}
For the first summand, note that since $tg \in \past(sg;g)$ and $\beta \le \alpha$,
\begin{eqnarray*}
H(T_{sg}^{-1}\alpha | T_g^{-1}\alpha) &\ge& H(T_{sg}^{-1}\alpha | T_g^{-1}\alpha \vee T_{tg}^{-1}\beta)\\
&\ge&  H\Big(T_{sg}^{-1}\alpha \big|\bigvee_{f \in \past(sg;g)} T_f^{-1}\alpha\Big) = H(T_{sg}^{-1}\alpha | T_g^{-1}\alpha).
\end{eqnarray*}
Hence
\begin{eqnarray}\label{eqn:ms2}
H\Big(T_{sg}^{-1}\alpha \big|\bigvee_{f \in \past(sg;g)} T_f^{-1}\alpha\Big) = H(T_{sg}^{-1}\alpha | T_g^{-1}\alpha \vee T_{tg}^{-1}\beta).
\end{eqnarray}
 For the second summand above, note that $\{g,tg\} \subset \past(sg;g)$ and $\{sg\} \cup  \past(sg;g) \subset \past(tsg;sg)$. Hence,
\begin{eqnarray*}
H(T_{tsg}^{-1}\alpha | T_{sg}^{-1}\alpha) &\ge& H(T_{tsg}^{-1}\alpha | T_{sg}^{-1}\alpha \vee T_{tg}^{-1}\beta \vee T_g^{-1}\alpha)\\
&\ge &H\Big(T_{tsg}^{-1}\alpha \big|T_{sg}^{-1}\alpha \vee \bigvee_{f \in \past(sg;g)} T_f^{-1}\alpha\Big)\\
&\ge& H\Big(T_{tsg}^{-1}\alpha \big|  \bigvee_{f \in \past(tsg;sg)}  T_f^{-1}\alpha\Big) = H(T_{tsg}^{-1}\alpha | T_{sg}^{-1}\alpha) .
\end{eqnarray*}
Hence
$$H\Big(T_{tsg}^{-1}\alpha ~\big|~ T_{sg}^{-1}\alpha \vee \bigvee_{f \in \past(sg;g)} T_f^{-1}\alpha\Big) = H(T_{tsg}^{-1}\alpha | T_{sg}^{-1}\alpha \vee T_{tg}^{-1}\beta \vee T_g^{-1}\alpha).$$
By the previous lemma this implies
\begin{eqnarray}\label{eqn:ms3}
H\Big(T_{tsg}^{-1}\beta ~\big|~ T_{sg}^{-1}\alpha \vee \bigvee_{f \in \past(sg;g)} T_f^{-1}\alpha\Big) = H(T_{tsg}^{-1}\beta | T_{sg}^{-1}\alpha \vee T_{tg}^{-1}\beta \vee T_g^{-1}\alpha).
\end{eqnarray}
Equations \ref{eqn:ms1}, \ref{eqn:ms1.5}, \ref{eqn:ms2} and \ref{eqn:ms3} imply
\begin{eqnarray*}
H\Big(T_{sg}^{-1}\sigma ~\big|~ \bigvee_{f \in \past(sg;g)} T_f^{-1}\sigma\Big) &=&H(T_{sg}^{-1}\alpha | T_g^{-1}\alpha \vee T_{tg}^{-1}\beta) + H(T_{tsg}^{-1}\beta | T_{sg}^{-1}\alpha \vee T_{tg}^{-1}\beta \vee T_g^{-1}\alpha)\\
&=& H(T_{sg}^{-1}\alpha \vee T_{tsg}^{-1}\beta | T_g^{-1}\alpha \vee T_{tg}^{-1}\beta)=H(T_{sg}^{-1}\sigma|T_g^{-1}\sigma).
\end{eqnarray*}

{\bf Case 2}. Assume $s= t$. Then
\begin{eqnarray*}
\bigvee_{f \in \past(sg;g)} T_f^{-1}\sigma &=& \bigvee_{f \in \past(sg;g)} T_f^{-1}\alpha \vee T_f^{-1}T_t^{-1}\beta = T_{tg}^{-1}\beta \vee \bigvee_{f \in \past(sg;g)} T_f^{-1}\alpha.
\end{eqnarray*}
Since $\beta \le \alpha$,
\begin{eqnarray}
H\Big(T_{sg}^{-1}\sigma \big| \bigvee_{f \in \past(sg;g)} T_f^{-1}\sigma\Big) &=& H\Big(T_{sg}^{-1}\sigma\big| T_{tg}^{-1}\beta \vee \bigvee_{f \in \past(sg;g)} T_f^{-1}\alpha\Big)\\
&=&  H\Big(T_{sg}^{-1}\alpha \big| T_{tg}^{-1}\beta \vee \bigvee_{f \in \past(sg;g)} T_f^{-1}\alpha\Big)\label{eqn:ms4}\\
&& + H\Big(T_{tsg}^{-1}\beta\big|T_{sg}^{-1}\alpha \vee \bigvee_{f \in \past(sg;g)} T_f^{-1}\alpha\Big).\label{eqn:ms4.5}
\end{eqnarray}
We claim that
\begin{eqnarray}\label{eqn:ms5}
H\Big(T_{tsg}^{-1}\beta \big| T_{sg}^{-1}\alpha \vee \bigvee_{f \in \past(sg;g)} T_f^{-1}\alpha\Big) = H(T_{tsg}^{-1}\beta | T_{sg}^{-1}\alpha \vee T_{tg}^{-1}\beta \vee T_g^{-1}\alpha).
\end{eqnarray}
This follows from the same argument used to prove equation \ref{eqn:ms3} except in one detail: $\{g,tg\} \nsubseteq \past(sg;g)$ this time. However, since $s=t$ and $\beta \le \alpha$,  it is still true that
$$H(T_{tsg}^{-1}\alpha | T_{sg}^{-1}\alpha \vee T_{tg}^{-1}\beta \vee T_g^{-1}\alpha)
\ge H\Big(T_{tsg}^{-1}\alpha \big|T_{sg}^{-1}\alpha \vee \bigvee_{f \in \past(sg;g)} T_f^{-1}\alpha\Big).$$
The rest of the proof of equation \ref{eqn:ms5} is the same as the proof of equation \ref{eqn:ms3}.  Next,
\begin{eqnarray}
H\Big(T_{sg}^{-1}\alpha \big| T_{tg}^{-1}\beta \vee \bigvee_{f \in \past(sg;g)} T_f^{-1}\alpha\Big) &=& H\Big(T_{sg}^{-1}\alpha \big| \bigvee_{f \in \past(sg;g)} T_f^{-1}\alpha\Big)\\
&& -H\Big( T_{tg}^{-1}\beta \big|  \bigvee_{f \in \past(sg;g)} T_f^{-1}\alpha\Big)\\
&=&  H(T_{sg}^{-1}\alpha | T_g^{-1}\alpha) - H( T_{tg}^{-1}\beta |T_g^{-1}\alpha)\\
&=& H(T_{sg}^{-1}\alpha| T_g^{-1}\alpha \vee T_{tg}^{-1}\beta).\label{eqn:ms6}
\end{eqnarray}
The second equality above follows from the previous lemma and the fact that $$H\Big( T_{tg}^{-1}\alpha\big|  \bigvee_{f \in \past(sg;g)} T_f^{-1}\alpha\Big) = H( T_{tg}^{-1}\alpha|  T_g^{-1}\alpha).$$ The third equality uses that $s=t$ so $T_{sg}^{-1}\alpha \ge T_{tg}^{-1}\beta$.
Equations \ref{eqn:ms4}, \ref{eqn:ms4.5}, \ref{eqn:ms5} and \ref{eqn:ms6} imply
\begin{eqnarray*}
H\Big(T_{sg}^{-1}\sigma ~\big|~ \bigvee_{f \in \past(sg;g)} T_f^{-1}\sigma\Big) &=&H(T_{sg}^{-1}\alpha | T_g^{-1}\alpha \vee T_{tg}^{-1}\beta) + H(T_{tsg}^{-1}\beta | T_{sg}^{-1}\alpha \vee T_{tg}^{-1}\beta \vee T_g^{-1}\alpha)\\
&=& H(T_{sg}^{-1}\alpha \vee T_{tsg}^{-1}\beta | T_g^{-1}\alpha \vee T_{tg}^{-1}\beta)=H(T_{sg}^{-1}\sigma|T_g^{-1}\sigma).
\end{eqnarray*}
\end{proof}

\begin{lem}
Let $(T,X,\mu,\alpha)$ be a Markov process. Let $\beta \le \alpha.$ Let $\beta^G$ be the smallest $G$-invariant $\sigma$-algebra containing the atoms of $\beta$. Then for every $s\in S$ and $g\in G$,
$$H\Big(T_{sg}^{-1}\alpha \big|  \bigvee_{f \in \past(sg;g)} T_f^{-1}\alpha \vee \beta^G\Big) = H(T_{sg}^{-1}\alpha | T_g^{-1}\alpha \vee \beta^G).$$
\end{lem}

\begin{proof}
Since
\begin{eqnarray*}
H\Big(T_{sg}^{-1}\alpha\big| \bigvee_{f \in \past(sg;g)} T_f^{-1}\alpha \vee \bigvee_{f \notin \past(sg;g)} T_f^{-1}\beta\Big)&=&H\Big(T_{sg}^{-1}\alpha \big|  \bigvee_{f \in \past(sg;g)} T_f^{-1}\alpha \vee \beta^G\Big)\\
&\le& H(T_{sg}^{-1}\alpha | T_g^{-1}\alpha \vee \beta^G)\\
&\le& H\Big(T_{sg}^{-1} \alpha \big| T_g^{-1}\alpha \vee \bigvee_{f \notin \past(sg;g)} T_f^{-1}\beta\Big),
\end{eqnarray*}
it suffices to show that
$$H\Big(T_{sg}^{-1}\alpha\big| \bigvee_{f \in \past(sg;g)} T_f^{-1}\alpha \vee \bigvee_{f \notin \past(sg;g)} T_f^{-1}\beta\Big) = H\Big(T_{sg}^{-1} \alpha \big| T_g^{-1}\alpha \vee \bigvee_{f \notin \past(sg;g)} T_f^{-1}\beta\Big).$$

To prove this, it suffices to show that for every left-connected (definition \ref{defn:leftcayley}) finite set $F \subset G-\past(sg;g)$ with $sg\in F$,
$$H\Big(T_{sg}^{-1}\alpha\big| \bigvee_{f \in \past(sg;g)} T_f^{-1}\alpha \vee \bigvee_{f \in F} T_f^{-1}\beta\Big) = H\Big(T_{sg}^{-1} \alpha \big| T_g^{-1}\alpha \vee \bigvee_{f \in F} T_f^{-1}\beta\Big).$$
Equivalently,
\begin{eqnarray*}
&&H\Big(T_{sg}^{-1}\alpha\vee \bigvee_{f \in F} T_f^{-1}\beta\big| \bigvee_{f \in \past(sg;g)} T_f^{-1}\alpha \Big) - H\Big( \bigvee_{f \in F} T_f^{-1}\beta\big| \bigvee_{f \in \past(sg;g)} T_f^{-1}\alpha \Big)\\
&&=H\Big(T_{sg}^{-1}\alpha\vee \bigvee_{f \in F} T_f^{-1}\beta\big|T_g^{-1}\alpha \Big) - H\Big( \bigvee_{f \in F} T_f^{-1}\beta\big| T_g^{-1}\alpha \Big)
\end{eqnarray*}
Thus, it suffices to prove the following two statements:
\begin{eqnarray*}
H\Big(T_{sg}^{-1}\alpha\vee \bigvee_{f \in F} T_f^{-1}\beta\big| \bigvee_{f \in \past(sg;g)} T_f^{-1}\alpha \Big) &=&H\Big(T_{sg}^{-1}\alpha\vee \bigvee_{f \in F} T_f^{-1}\beta\big|T_g^{-1}\alpha \Big)\\
H\Big( \bigvee_{f \in F} T_f^{-1}\beta\big| \bigvee_{f \in \past(sg;g)} T_f^{-1}\alpha \Big) &=& H\Big( \bigvee_{f \in F} T_f^{-1}\beta\big| T_g^{-1}\alpha \Big).
\end{eqnarray*}
By lemma \ref{lem:additive}, it suffices to prove
\begin{eqnarray}\label{eqn:mc1}
H\Big( \bigvee_{f \in F} T_f^{-1}\alpha\big| \bigvee_{f \in \past(sg;g)} T_f^{-1}\alpha \Big) = H\Big( \bigvee_{f \in F} T_f^{-1}\alpha\big| T_g^{-1}\alpha \Big).
 \end{eqnarray}
 We will prove this by induction on $|F|$. If $|F|=1$, then this follows immediately from the definition of Markov processes. If $|F|>1$ then there exists $f_0 \in F$ and $t\in S$ such that $tf_0 \in F$, $tf_0\ne sg$ and $F':=F-\{tf_0\}$ is left-connected. So,
 \begin{eqnarray}\label{eqn:mc2}
 H\Big( \bigvee_{f \in F} T_f^{-1}\alpha\big| \bigvee_{f \in \past(sg;g)} T_f^{-1}\alpha \Big) &=& H\Big( \bigvee_{f \in F'} T_f^{-1}\alpha\big| \bigvee_{f \in \past(sg;g)} T_f^{-1}\alpha \Big)\\
 &&+ H\Big( T_{tf_0}^{-1}\alpha\big|  \bigvee_{f \in F'} T_f^{-1}\alpha \vee \bigvee_{f \in \past(sg;g)} T_f^{-1}\alpha \Big).
 \end{eqnarray}
 By induction,
  \begin{eqnarray}\label{eqn:mc3}
 H\Big( \bigvee_{f \in F'} T_f^{-1}\alpha\big| \bigvee_{f \in \past(sg;g)} T_f^{-1}\alpha \Big) = H\Big( \bigvee_{f \in F'} T_f^{-1}\alpha\big| T_g^{-1}\alpha \Big).
 \end{eqnarray}
Since $F' \cup \past(sg;g) \subset \past(tf_0;f_0)$,
 \begin{eqnarray}
 H\Big( T_{tf_0}^{-1}\alpha\big|  \bigvee_{f \in \past(tf_0;f_0)} T_f^{-1}\alpha \Big)
 &\le& H\Big( T_{tf_0}^{-1}\alpha\big|  \bigvee_{f \in F'} T_f^{-1}\alpha \vee \bigvee_{f \in \past(sg;g)} T_f^{-1}\alpha \Big) \\
 &\le& H\Big( T_{tf_0}^{-1}\alpha\big|  \bigvee_{f \in F' \cup \{g\}} T_f^{-1}\alpha \Big) \\
 &\le& H( T_{tf_0}^{-1}\alpha|  T_{f_0}^{-1}\alpha ) \\
&=& H\Big( T_{tf_0}^{-1}\alpha\big|  \bigvee_{f \in \past(tf_0;f_0)} T_f^{-1}\alpha \Big).\label{eqn:mc4}
 \end{eqnarray}
So equality holds throughout. Equations \ref{eqn:mc2}, \ref{eqn:mc3} and \ref{eqn:mc4} imply
\begin{eqnarray*}
 H\Big( \bigvee_{f \in F} T_f^{-1}\alpha\big| \bigvee_{f \in \past(sg;g)} T_f^{-1}\alpha \Big)&=& H\Big( \bigvee_{f \in F'} T_f^{-1}\alpha\big| T_g^{-1}\alpha \Big)+H\Big( T_{tf_0}^{-1}\alpha\big| \bigvee_{f \in F' \cup \{g\}} T_f^{-1}\alpha\Big)\\
 &=& H\Big(\bigvee_{f \in F} T_f^{-1}\alpha \big| T_g^{-1}\alpha\Big).
 \end{eqnarray*}
This proves equation \ref{eqn:mc1} and hence finishes the lemma.
\end{proof}

To simplify notation, we write $F(\alpha|\sF)$ for $F(T,\alpha|\sF)$ when $T$ is clear. Similar statements apply to $f(\alpha|\sF)$, $F_*(\alpha|\sF)$, etc.
\begin{lem}\label{lem:f=F}
Let $(T,X,\mu,\alpha)$ be a Markov process. Let $\beta$ be a coarsening of $\alpha$. Then for any splitting $\sigma$ of $\alpha$, $F(\alpha|\beta^G)=F(\sigma|\beta^G)$.
\end{lem}

\begin{proof}
By lemma \ref{lem:splitting1} it suffices to consider the special case in which $\sigma$ is a simple splitting of $\alpha$. So there is a $t \in S$ such that $\alpha \le \sigma \le \alpha \vee T_t^{-1}\alpha$. By proposition \ref{prop:splittingmonotone}, $F(\alpha|\beta^G) \le F(\sigma|\beta^G) \le F(\alpha \vee T_t^{-1}\alpha|\beta^G)$. Hence it suffices to show that $F(\alpha|\beta^G)=F(\alpha \vee T_t^{-1}\alpha|\beta^G)$. If $G$ is a group rather than a semigroup then by $G$-invariance,
$$F(\alpha\vee T_t^{-1}\alpha|\beta^G) = F(\alpha\vee T_{t^{-1}}^{-1}\alpha|\beta^G).$$
So, without loss of generality, we may assume that $t =s_r \in S$.

We claim that
\begin{eqnarray}\label{eqn:mf1}
F(\alpha \vee T_t^{-1}\alpha|\beta^G) &=& F(\alpha|\beta^G) + (1-2r)H(T_t^{-1}\alpha|\alpha \vee \beta^G)\\
&& + \sum_{i=1}^r H( T_t^{-1}\alpha  \vee T_{s_i}^{-1}T_t^{-1}\alpha|\beta^G \vee  \alpha \vee T_{s_i}^{-1}\alpha).
\end{eqnarray}
To see this, in the formula for $F(\alpha \vee T_t^{-1}\alpha |\beta^G)$, replace $H(\alpha \vee t\alpha|\beta^G)$ with $H(\alpha|\beta^G) + H( T_t^{-1}\alpha |\alpha \vee \beta^G)$ and for each $s \in \{s_1,\ldots,s_r\}$ replace $$H(\alpha \vee T_t^{-1}\alpha \vee T_s^{-1}\alpha \vee T_s^{-1}T_t^{-1}\alpha|\beta^G)$$ with  $$H(\alpha \vee T_s^{-1}\alpha|\beta^G) + H( T_t^{-1}\alpha  \vee T_s^{-1}T_t^{-1}\alpha|\beta^G \vee  \alpha \vee T_s^{-1}\alpha).$$
Collecting terms implies the claim.

Note that for any $s\in S$,
\begin{eqnarray}
H( T_t^{-1}\alpha  \vee T_{s}^{-1}T_t^{-1}\alpha|\beta^G \vee  \alpha \vee T_{s}^{-1}\alpha)&=& H( T_t^{-1}\alpha |\beta^G \vee  \alpha \vee T_{s}^{-1}\alpha)\\
&& + H( T_{s}^{-1}T_t^{-1}\alpha | \beta^G \vee  \alpha \vee T_{s}^{-1}\alpha \vee  T_t^{-1}\alpha).
\end{eqnarray}
If $s=t=s_r$ then the above quantity equals $0 + H( T_{t^2}^{-1}\alpha | \beta^G \vee  \alpha \vee T_{t}^{-1}\alpha).$  By the previous lemma, this equals $H( T_{t^2}^{-1}\alpha | \beta^G \vee T_{t}^{-1}\alpha) = H( T_{t}^{-1}\alpha | \beta^G \vee \alpha).$ Now substitute this into equation \ref{eqn:mf1} to obtain
\begin{eqnarray}\label{eqn:mf2}
&&F(\alpha \vee T_t^{-1}\alpha|\beta^G) - F(\alpha|\beta^G)\\
&&  = \sum_{i=1}^{r-1} H( T_t^{-1}\alpha  \vee T_{s_i}^{-1}T_t^{-1}\alpha|\beta^G \vee  \alpha \vee T_{s_i}^{-1}\alpha) - 2H(T_t^{-1}\alpha|\alpha \vee \beta^G).
\end{eqnarray}
If $s\ne t$ then the previous lemma implies
\begin{eqnarray*}
H( T_t^{-1}\alpha  \vee T_{s}^{-1}T_t^{-1}\alpha|\beta^G \vee  \alpha \vee T_{s}^{-1}\alpha)&=& H( T_t^{-1}\alpha |\beta^G \vee  \alpha \vee T_{s}^{-1}\alpha)\\
&& + H( T_{s}^{-1}T_t^{-1}\alpha | \beta^G \vee  \alpha \vee T_{s}^{-1}\alpha \vee  T_t^{-1}\alpha)\\
&=&  H( T_t^{-1}\alpha |\beta^G \vee  \alpha) +H( T_{s}^{-1}T_t^{-1}\alpha | \beta^G \vee   T_{s}^{-1}\alpha)\\
&=& 2H(T_t^{-1}\alpha|\alpha \vee \beta^G).
\end{eqnarray*}
Equation \ref{eqn:mf2} now implies $F(\alpha \vee T_t^{-1}\alpha|\beta^G) = F(\alpha|\beta^G)$ as claimed.
\end{proof}

We can now prove theorem \ref{thm:f=F}.

\begin{proof}[Proof of theorem \ref{thm:f=F}]
This first equality follows from the previous lemma and definition \ref{defn:f}. To prove the second equality, first note that for any $s\in S$, 
$$h(T_s, \alpha|\beta^G) = \lim_{n\to\infty} H\Big(T_s^{-n-1}\alpha |\beta^G \vee \bigvee_{i=0}^n T_s^{-i}\alpha\Big).$$
Hence,
\begin{eqnarray*}
H(\alpha \vee T_s^{-1} \alpha|\beta^G) -H(\alpha|\beta^G) &=& H(T_s^{-1}\alpha |\alpha \vee \beta^G)\\
&=& H(T_s^{-n-1}\alpha |T_s^{-n}\alpha \vee \beta^G)\\
&\ge& H\Big(T_s^{-n-1}\alpha \big| \beta^G \vee \bigvee_{i=0}^n T_s^{-i}\alpha\Big)\\
&\ge& H\Big(T_s^{-n-1}\alpha \big| \beta^G \vee \bigvee_{f \in \past(s^{n+1};s^n)} T_f^{-1}\alpha\Big)\\
&=& H(T_s^{-n-1}\alpha |T_s^{-n}\alpha \vee \beta^G) = H(T_s^{-1}\alpha|\alpha \vee \beta^G).
\end{eqnarray*}
Thus, equality holds throughout. Hence
$$h(T_s, \alpha|\beta^G) =H(\alpha \vee T_s^{-1} \alpha|\beta^G) -H(\alpha|\beta^G).$$
We now have
\begin{eqnarray*}
F_*(\alpha|\beta^G) &=& (1-r)H(\alpha|\beta^G) +\sum_{i=1}^r h(T_{s_i}, \alpha|\beta^G) \\
&=& (1-2r)H(\alpha|\beta^G) +\sum_{i=1}^r H(\alpha \vee T_s^{-1}\alpha|\beta^G)= F(\alpha|\beta^G).
\end{eqnarray*}
This proves the second equality in the statement. By proposition \ref{prop:splitting}, $\alpha^n$ is a splitting of $\alpha$. By lemma \ref{lem:splitting1}, $(T,X,\mu,\alpha^n)$ is a Markov process. Thus by the above, $F_*(\alpha^n|\beta^G)=F(\alpha^n|\beta^G)=f(\alpha|\beta^G)$ for all $n\ge 0$. Take the infimum over all $n$ to see that $f_*(\alpha|\beta^G)=F(\alpha|\beta^G)$.
\end{proof}

\section{Markov Chains}\label{sec:markovchains}


The purpose of this section is to develop a constructive approach to Markov processes through transition matrices and symbolic dynamics. This will be used later to prove $f=f_*$ in general.

\subsection{The existence theorem}

\begin{defn}
An {\bf ordered process} is a quadruple $(T,X,\mu,\alpha)$ where $(T,X,\mu)$ is a $G$-system and $\alpha=(A_1,A_2,\ldots)$ is an ordered partition. Two ordered processes $(T,X,\mu,\alpha)$, $(S,Y,\nu,\beta)$ are {\bf isomorphic} (as ordered processes) if there is a measure-conjugacy $\phi:X \to Y$ that maps the $i$-th atom of $\alpha$ to the $i$-th atom of $\beta$ for all $i\ge 1$.
\end{defn}

\begin{defn}
Let $\bX=(T,X,\mu,\alpha)$ and $\bY=(U,Y,\nu,\beta)$ be two ordered processes with $\alpha=(A_1,A_2,\ldots)$ and $\beta=(B_1,B_2,\ldots)$. For $n \ge 0$ let
$$d_1(\bX,\bY)=\sum_{s \in S} \sum_{i,j = 1}^\infty \big|\mu(A_{i} \cap T_s^{-1}A_{j}) - \nu(B_{i} \cap T_s^{-1}B_{j})\big|.$$
Here we are following the convention that if, for example, $\alpha=(A_1,\ldots,A_n)$ is finite then $A_i:=\emptyset$ for all $i>n$. $d_1$ is symmetric and satisfies the triangle inequality but it is not a distance function since two nonisomorphic processes could be at distance zero from each other.
\end{defn}

The main result of this section is:

\begin{thm}\label{thm:existence}
Let $\bY=(U,Y,\nu,\beta)$ be an ordered process. Then there exists a Markov process $\bX=(T,X,\mu,\alpha)$ such that $d_1(\bX,\bY)=0$. Moreover, $\bY$ is unique up to isomorphism (as an ordered process).
\end{thm}

\subsection{Symbolic dynamics notation}\label{sec:notation}

If $K$ is any topological space then $K^G$ denotes the set of all functions $x:G \to K$. It can also be thought of as the product space $K^G = \prod_{g \in G} K$ and hence is endowed with the product topology. In most of the applications of this paper, $K$ is either finite or countably infinite. In these cases, it is implicitly assumed that $K$ has the discrete topology and this induces the product topology on $K^G$. The {\bf canonical action} of $G$ on $K^G$ is defined by $T_gx(f)=x(fg) \forall f,g\in G, x\in K^G$. The {\bf canonical partition} of $K^G$ is $\alpha=\{A_k~|~k \in K\}$ where $A_k=\{x \in K^G | x(e)=k\}$.

A measure $\mu$ on $K^G$ is invariant if $\mu(T^{-1}_gE)=\mu(E)$ for all Borel $E \subset K^G$ and $g\in G$. Let $M(K^G)$ denote the space of all invariant Borel probability measures $\mu$ on $K^G$. The {\bf weak* topology} on $M(K^G)$ is defined as follows. We say that a sequence $\{\mu_n\}_{n=1}^\infty \subset M(K^G)$ converges to $\mu \in M(K^G)$ in the weak* topology if and only if for every continuous function $f:K^G \to \R$, $\lim_{n\to\infty} \int f ~d\mu_n =\int f ~d\mu$. Equivalently, $\lim_{n \to \infty} \mu_n = \mu$ (weak*) if and only if for every $m\ge 0$ and every $A \in \alpha^m$, $\lim_{n\to\infty} \mu_n(A) = \mu(A)$.

\subsection{Transition Systems}
\begin{defn}
Let $K$ be a finite or countably infinite set. A {\bf stochastic matrix} $P$ with {\bf state space $K$} is a $K \times K$ matrix $P=(P_{ij})$ such that
\begin{itemize}
\item $0\le P_{ij} \le 1$ for all $i,j$,
\item for each $i$, $\sum_{j\in K} P_{ij} = 1$.
\end{itemize}
A $1 \times K$ vector $\pi$ is a {\bf probability vector} if its entries are nonnegative and sum to one. If, in addition, $\pi P=\pi$ then $\pi$ is a  {\bf steady state vector} for $P$.
\end{defn}

\begin{defn}
A {\bf transition system} for $(G,S)$ is a collection of stochastic matrices $\{P^s\}_{s \in S}$ and a probability vector $\pi$. It is an {\bf invariant transition system} if the following hold.
\begin{itemize}
\item For all $s\in S$, $\pi$ is a steady state vector for $P^s$.
\item If $G$ is a group rather than a semigroup then for all $s \in S$, $i,j \in K$, $\pi_i P^{s^{-1}}_{ij} = \pi_j P^s_{ji}$. Just to be careful, note that $P^{s^{-1}}$ is not the inverse of $P^s$. It is equals $P^t$ where $t=s^{-1}$. 
\end{itemize}


Even if $G \cong \Z$, this definition differs from the classical case in a minor detail. Typically, only one transition matrix is given. But the above definition requires two: $P^s$ and $P^{s^{-1}}$. Of course, the second condition above implies that $P^{s^{-1}}$ is determined by $P^s$ so the two definitions are really equivalent. This redundancy will make forthcoming arguments a little simpler.
\end{defn}

\begin{defn}\label{defn:markovchain}
The {\bf Markov chain} over $G$ induced by the transition system $\P:=(\{P^s\}_{s\in S}, \pi)$ is the $G$-indexed set of random variables $(X_g)_{g \in G}$ satisfying the following conditions:
\begin{itemize}
\item The distribution of $X_e$ equals $\pi$. I.e., for any $k \in K$, the probability that $X_e=k$ equals $\pi_k$. Formally, $Pr(X_e=k)=\pi_k$.
\item Let $g \in G$ and $s \in S$ be such that $|sg|=|g|+1$ where $|\cdot|$ denotes word length. Let $f_1,\ldots,f_n \in \past(sg;g)-\{g\}$. Then for any $k, k_0,\ldots,k_n \in K$,
$$Pr(X_{sg}=k | X_{g}=k_0, X_{f_1}=k_1,\ldots,X_{f_n}=k_n) = Pr(X_{sg}=k | X_{g}=k_0)=P^s_{k_0,k}.$$
\end{itemize}
It is an {\bf invariant Markov chain} if $\P$ is invariant.
\end{defn}

\begin{defn}
For any measure $\mu$ on $X$ and any Borel sets $A, B \subset X$ with $\mu(B)>0$ define
$$\mu(A|B) = \frac{\mu(A \cap B)}{\mu(B)}.$$
\end{defn}

\begin{defn}
Let $(X_g)_{g\in G}$ be defined as above. Define the random function $x:G \to K$ by $x(g)=X_g$. Let $\mu$ be the probability measure on $K^G$ equal to the law of $x$ (i.e., for any Borel $E\subset K^G$, $\mu(E)$ is the probability that $x$ is contained in $E$). 

We say that $(T, K^G,\mu,\alpha)$ is the process induced by the transition system $\P:=(\{P^s\}_{s\in S}, \pi)$ (where $T$ is the canonical action of $G$ on $K^G$ and $\alpha$ is the canonical partition of $K^G$). Corollary \ref{cor:invariant} below shows that it is Markov.

The conditions on $(X_g)_{g\in G}$ stated above can be restated in terms of the measure $\mu$ as follows. For each $k\in K$, let $A_k=\{ y \in K^G~|~y(e)=k\}$. Then
\begin{itemize}
\item For all $k\in K$, $\mu(A_k)=\pi_k$,

\item Let $g \in G$ and $s \in S$ be such that $|sg|=|g|+1$. Let $f_1,\ldots,f_n \in \past(sg;g)-\{g\}$. Then for any $k, k_0,\ldots,k_n \in K$,
\begin{eqnarray*}
\mu\Big(T_{sg}^{-1}A_k \big| T^{-1}_gA_{k_0} \cap \bigcap_{i=1}^n T_{f_i}^{-1}A_{k_i}\Big) &=& \mu\big(T_{sg}^{-1}A_k \big| \cap T_g^{-1}A_{k_0}\big)= P^s_{k_0,k}.
\end{eqnarray*}
\end{itemize}
\end{defn}
In order to prove that $(T, K^G,\mu,\alpha)$ is a Markov process, we first need to prove that $\mu$ is $T$-invariant (when $\P$ is invariant). This is accomplished next. So fix an invariant transition system $\P:=(\{P^s\}_{s\in S}, \pi)$. For the next three lemmas, the identity element in $G$ is denoted by $id$.

\begin{defn}
Let $\Gamma_L$ be the left-Cayley graph of $(G,S)$ (definition \ref{defn:leftcayley}). If $e$ is an edge of $\Gamma_L$, let $e_-, e_+$ denote the endpoints of $e$ where $e_-$ is the vertex that is closest to the identity element in $\Gamma_L$. If $F \subset G$ is any set, let $E(F)$ denote the set of edges $e$ in $\Gamma_L$ that are directed from $e_- \in F$ to $e_+ \in F$.
\end{defn}

\begin{lem}\label{lem:mcformula}
Let $F \subset G$ be a finite left-connected set with $id \in F$. Let $z:F\to K$ be an arbitrary function and let
$$C=\{x \in K^G ~|~ x(g)=z(g) \forall g \in F\}$$
be the {\bf cylinder set induced by $z$ and $F$}. For each edge $e\in E(F)$, Let $p_z(e)=P^s_{ij}$ where $z(e_-)=i$ and $z(e_+)=j$ and $s \in S$ is such that $se_- = e_+$. Then
$$\mu(C) = \pi_{z(id)} \prod_{e \in E(F)} p_z(e).$$
\end{lem}
\begin{proof}
This is immediate from the definition.
\end{proof}

Our proof of invariance handles the group case separately from the semigroup case.
\begin{lem}\label{lem:invariance}
Suppose $G$ is a group. For all $g\in G$ and all Borel $E \subset K^G$, $\mu(E)=\mu(T_g^{-1}E)$. I.e., $\mu$ is $T_g$-invariant.
\end{lem}

\begin{proof}
Let $F, C, z$ be as in the previous lemma. Assume that $S \subset F$. Let $t \in S$. We will show that $\mu(C)=\mu(T_t^{-1}C)$. Let $t^{-1}z:Ft \to K$ be the function $(t^{-1}z)(ft)=z(f)$ for all $f \in F$. Then
\begin{eqnarray*}
T_t^{-1}C&=&\{x \in K^G ~|~ x(gt)=z(g) ~\forall g \in F\}\\
  &=&\{x \in K^G ~|~ x(g)=(t^{-1}z)(g) ~\forall g \in Ft\}.
\end{eqnarray*}
The previous lemma implies
\begin{eqnarray*}\label{lem:mc10}
\mu(T_t^{-1}C)=\pi_{z(t^{-1})} \prod_{e \in E(Ft)} p_{t^{-1}z}(e).
\end{eqnarray*}
If $e$ is an edge of $\Gamma_L$ then let $e\cdot t$ denote the edge with endpoints $e_-t$ and $e_+t$. 

{\bf Claim :} Either $(e_-,e_+)=(id,t^{-1})$ or $\big((e\cdot t)_-, (e\cdot t)_+\big) =(e_-t, e_+t)$.

To prove the claim, let $g\in G$, $s\in S$ be such that $(e_-, e_+)=(g,sg)$.  Let $j$ denote the path in $\Gamma_L$ from $id$ to $sg$. Then $j\cdot t$ is the path in $\Gamma_L$ from $t$ to $sgt$. If $|g| \ge 1$ then this path has length at least 2. This implies $|t| \le |gt| \le |sgt|$. I.e., $(e\cdot t)_-=e_-t$ and $(e \cdot t)_+ = e_+t$. The case $|g|=0$ (i.e., $g=id$) is obvious. This proves the claim.

The claim implies that if $e \in E(F)$ is such that $(e_-,e_+) \ne (id,t^{-1})$ then $p_{t^{-1}z}(e\cdot t) = p_z(e)$. So if we let $e_*$ be the edge from $id$ to $t^{-1}$ then
$$\mu(T_t^{-1}C)=\pi_{z(t^{-1})} p_{t^{-1}z}(e_*\cdot t) \prod_{e \in E(F)-\{e_*\}} p_{z}(e).$$

Let $i=z(id)$ and $j=z(t^{-1})$. By definition of $p$ and the definition of an invariant transition system,
\begin{eqnarray*}
\mu(T_t^{-1}C)&=& \pi_j P^t_{ji}  \prod_{e \in E(F)-\{e_*\}} p_{z}(e)= \pi_i P^{t^{-1}}_{ij}  \prod_{e \in E(F)-\{e_*\}} p_{z}(e) = \mu(C).
\end{eqnarray*}
Since this is true for all cylinder sets $C$ whose domain contains $S$, it is true for all cylinder sets (since any cylinder set is a disjoint union of such sets). Since the cylinder sets generate the Borel $\sigma$-algebra of $K^G$, it follows that $\mu(T_t^{-1}E)=\mu(E)$ for all Borel sets $E \subset K^G$. Since this is true for all $t\in S$ and $S$ generates $G$, it follows that $\mu$ is $T_g$-invariant for all $g\in G$.
\end{proof}

\begin{lem}\label{lem:semigroupinvariance}
Suppose $G$ is a semigroup. For all $g\in G$ and all Borel $E \subset K^G$, $\mu(E)=\mu(T_g^{-1}E)$. I.e., $\mu$ is $T_g$-invariant.
\end{lem}

\begin{proof}
Let $F, C, z$ be as in the lemma \ref{lem:mcformula}. Let $t \in S$. For each $k \in K$, let $z_k:\{id\} \cup Ft \to K$ be defined by $z_k(ft)=z(f)$ if $f \in F$ and $z_k(id)=k$. Let $C_k=\{x \in K^G ~|~ x(g)=z_k(g) \forall g \in \{id\}\cup F\}$. Since $T_t^{-1}C$ is the disjoint union of $C_k$ over $k\in K$, it follows from lemma \ref{lem:mcformula} that
$$\mu(T^{-1}_t C) = \sum_{k\in K} \mu(C_k) = \sum_{k\in K}\pi_k \prod_{e \in E(\{id\} \cup Ft)} p_{z_k}(e).$$
Let $e_*$ be the edge from $id$ to $t$. Then $p_{z_k}(e_*)=P^t_{kl}$ where $z(id)=l$. If $e \in F$ then $e\cdot t \in Ft$ and
$\big( (e\cdot t)_-, (e\cdot t)_+ \big) = \big(e_-t,e_+t\big)$. Hence $p_{z_k}(e \cdot t)=p_z(e)$. Also, $E(\{id\}\cup Ft) = E(F) \cup \{e_*\}$. Thus,
$$\mu(T^{-1}_t C) =  \sum_{k\in K} \pi_k P^t_{kl} \prod_{e \in E( F)} p_{z}(e) = \pi_l \prod_{e \in E( F)} p_{z}(e)= \mu(C).$$
The second equality follows from the assumption that $\pi$ is a steady state vector for $P^t$.

Since the cylinder sets generate the Borel $\sigma$-algebra of $K^G$, it follows that $\mu(T_t^{-1}E)=\mu(E)$ for all Borel sets $E \subset K^G$. Since this is true for all $t\in S$ and $S$ generates $G$, it follows that $\mu$ is $T_g$-invariant for all $g\in G$.
\end{proof}

\begin{cor}\label{cor:invariant}
Any process $(T,K^G,\mu,\alpha)$ induced by an invariant transition system $\P$ is Markov.
\end{cor}

\begin{proof}
This follows immediately from the previous two lemmas.
\end{proof}

\begin{cor}
If $(T,K^G,\mu,\alpha)$ is induced by an invariant transition system $\P=(\{P^s\}_{s\in S},\pi)$ then
$$f(T) = (2r-1)\sum_{i\in K} \pi_i\log(\pi_i)-\sum_{s\in S_+} \sum_{i,j \in K} \pi_iP^s_{ij}\log(\pi_i P^s_{ij}) .$$
Here $S_+=\{s_1,\ldots,s_r\}$.
\end{cor}

\begin{proof}
This follows from the previous corollary and theorem \ref{thm:f=F}.
\end{proof}

\begin{proof}[Proof of theorem \ref{thm:existence}]
Let $\beta=(B_1, B_2,\ldots)$ and $K=\N$. Let $\pi$ be the $1 \times K$-vector defined by $\pi_k=\nu(B_k)$. Let $P^s_{ij}=\nu(U_s^{-1}B_j~|~B_i)$. It is a simple exercise (using the $U$-invariance of $\nu$) to check that $\P=(\{P^s\}_{s\in S}, \pi)$ is an invariant transition system. Let $\bX=(T,K^G,\mu,\alpha)$ be the Markov process induced by $\P$. It is automatic that $d_1(\bX,\bY)=0$. This proves existence. Uniqueness is trivial.
\end{proof}

\section{Examples}\label{sec:examples}

In this section, we give three examples of Markov chains over free groups; one related to the Wired Spanning Forest, to perfect matchings, and a third one with negative $f$-invariant. These are not used in the rest of the paper.

\subsection{The Wired Spanning Forest}\label{sec:example1}

The uniform spanning tree (UST) on a finite graph is a subgraph chosen uniformly at random among all spanning trees. In [Pe91], R. Pemantle answered a question of R. Lyons by showing that if $\sG$ is an infinite graph and if $\sG_1 \subset \sG_2 \subset \ldots$ is an exhaustion of $\sG$ by finite connected subgraphs, then the weak limit of the UST on $\sG_n$ exists. The limit is called the free spanning forest (FSF). In his proof, R. Pemantle introduced another model that is now called the wired spanning forest (WSF). It is defined as follows. As above, let $\sG_1 \subset \sG_2 \subset \ldots$ be an exhaustion of $\sG$ by finite connected subgraphs. Let $\sG_i^w$ be the graph $\sG_i$ with all of its boundary vertices identified (i.e., wired) to a single vertex. Then the WSF on $\sG$ is the weak limit of the UST on $\sG^w_i$ as $i \to \infty$. See [BLPS01] for a thorough study of the construction and properties of the FSF and WSF as well as references to other works on the subject.

Here we are interested in the WSF on the left-Cayley graph $\Gamma=\Gamma_L$ of the group $G=\langle s_1,\ldots,s_r\rangle$. We will describe it as a Markov chain over $G$ with state space $S=\{s_1^{\pm 1},\ldots, s_r^{\pm 1}\}$. But before this, we give a little intuition as to what we are doing.

Let $x:G \to S$ be a function. Let $F_x$ be the subgraph of $\Gamma$ defined as follows. An edge from $g$ to $sg$ is in $F_x$ if and only if either $x(g)=s$ or $x(sg)=s^{-1}$. It is automatic that $F_x$ is a spanning forest because the Cayley graph $\Gamma_L$ is a tree. Now, suppose $x$ satisfies the following condition: if $x(g) = s \in S$ then $x(sg)\ne s^{-1}$. In this case, $F_x$ has no finite components. The Markov measure $\mu$ on $S^G$ that we will define is maximally symmetric and has the property that if $x:G \to S$ is a random element drawn according to $\mu$ then $x$ satisfies the above condition so that $F_x$ has no finite components.

The transition system of the Markov chain is denoted here by $\P=(\{P^s\}_{s\in S}, \pi)$ as usual. In agreement with the above discussion, $P^s_{ss^{-1}}=0$ for all $s\in S$. The symmetry considerations lead to the following values for every $s\in S$.
$$\pi_s=\frac{1}{|S|},$$
$$P^s_{st} = \frac{1}{|S|-1} \textrm{ for all } t \ne s^{-1},$$
$$P^s_{ts^{-1}} = \frac{1}{|S|-1} \textrm{ for all } t \ne s,$$
$$P^s_{uv} = \frac{|S|-2}{(|S|-1)^2} \textrm{ for all } u,v \in S \textrm{ with } u\ne s, v\ne s^{-1}.$$
So, the $f$-value of this system is:
\begin{eqnarray*}
&&\frac{|S|}{2}\Big( \frac{2}{|S|}\log(|S|(|S|-1)) + \frac{|S|-2}{|S|}\log(\frac{(|S|-1)^2|S|}{|S|-2}\Big) - (|S|-1)\log(|S|)\\
&=& ( 1+ \frac{|S|-2}{2}-|S|+1)\log(|S|) +(|S|-1) \log(|S|-1) - \frac{|S|-2}{2}\log(|S|-2)\\
&=& (1-r)\log(2r) + (2r-1) \log(2r-1) + (1-r)\log(2r-2).
\end{eqnarray*}
Using Wilson's algorithm [Wi96], it can be proven that the random graph $F_x$ (where $x$ has law given by the above Markov measure) is the WSF. For a comparison, let $\sG_n$ be a connected graph on $n$ vertices such the random weak limit  of the sequence $\{\sG_n\}$ is a $2r$-regular tree (see [Ly05] for definitions). Improving on an earlier result of [Mc83], in [Ly05] it is proven that the exponential growth rate of the number of spanning trees in $\sG_n$ is exactly $(1-r)\log(2r) + (2r-1) \log(2r-1) + (1-r)\log(2r-2)$.

\subsection{Perfect Matchings}\label{sec:example2}

There is a natural random perfect matching on the left-Cayley graph $\Gamma=\Gamma_L$ of the free groupÊ $G=\langle s_1,\ldots ,s_r\rangle$. We will describe it as a Markov chain over $G$ with state space $S=\{s_1^{\pm 1},\ldots,s_r^{\pm 1}\}$. But before this, we give a little intuition as to what we are doing.

As in the previous example, let $x:G \to S$ be a function. Let $F_x$ be the subgraph of $\Gamma_L$ defined as follows. An edge from $g$ to $sg$ is in $F_x$ if and only if either $x(g)=s$ or $x(sg)=s^{-1}$. It is automatic that $F_x$ is a spanning forest because $\Gamma_L$ is a tree. Now, suppose $x$ satisfies the following condition: if $x(g) = s \in S$ then $x(sg)= s^{-1}$. In this case, every component of $F_x$ consists of a single edge. So $F_x$ is a perfect matching. The Markov measure $\mu$ on $S^G$ that we will define is maximally symmetric and has the property that if $x:G \to S$ is a random element drawn according to $\mu$ then $x$ satisfies the above condition so that $F_x$ is a perfect matching.

The transition system of the Markov chain is denoted here by $\P=(\{P^s\}_{s\in S}, \pi)$ as usual. In agreement with the above discussion, $P^s_{ss^{-1}}=1$ for all $s\in S$. Thus, $P^s_{st}=0$ for all $t\ne s^{-1}$. Imposition of maximal symmetry conditions leads to the following values for every $s\in S$.
$$\pi_s=\frac{1}{|S|},$$
$$P^s_{ts^{-1}} = 0 \textrm{ for all } t \ne s,$$
$$P^s_{uv} = \frac{1}{|S|-1} \textrm{ for all } u,v \in S \textrm{ with } u\ne s, v\ne s^{-1}.$$
So, the $f$-value of this system is:
\begin{eqnarray*}
&& -(1/2)\big(\sum_{s\in S} \sum_{i,j \in K} \pi_iP^s_{ij}\log(\pi_i P^s_{ij})\big) + (2r-1)\sum_{i\in K} \pi_i\log(\pi_i)\\
&&=(1/2)\log (|S|) + \Big(\frac{|S|-1}{2}\Big)\log(|S|(|S|-1)) - (2r-1)\log(|S|)\\
&&=-\Big(\frac{2r-2}{2}\Big) \log(2r) + \Big(\frac{2r-1}{2}\Big)\log(2r-1).
\end{eqnarray*}
For a comparison, let $\sG_{n,2r}$ be a graph chosen uniformly at random among all $2r$-regular graphs on $n$ vertices. In [BM86], it is proven that $\bE[M_n]$, the expected number of perfect matchings on $\sG_{n,2r}$ is asymptotic (as $n\to \infty$) to
$$\sqrt{2}e^{1/4}\exp\Big( -\Big(\frac{2r-2}{2}\Big) \log(2r)n + \Big(\frac{2r-1}{2}\Big)\log(2r-1)n\Big).$$

\subsection{A mixing Markov chain with negative $f$-invariant}\label{sec:example3}

\begin{prop}
If $G$ is a nonabelian free group then there exists a Markov process $(T,K^G,\mu,\alpha)$ such that $-\infty < f(T)<0$.
\end{prop}

\begin{proof}
Let $0\le \epsilon \le 1$ be given. Let $K$ be a two-element set. Let $\pi=[\frac{1}{2} \frac{1}{2}]$. For each $s\in S$, let
\begin{displaymath}
P^s=\left[ \begin{array}{cc}
\epsilon & 1-\epsilon \\
1-\epsilon & \epsilon
\end{array} \right].
\end{displaymath}
It is easy to check that $\P=(\{P^s\}_{s\in S}, \pi)$ is an invariant transition system for all $\epsilon \in [0,1]$. Let $(T,K^G,\mu_\epsilon,\alpha)$ be the induced Markov process. Its $f$-value, denoted $f(T,\mu_\epsilon)$, varies continuously with $\epsilon$. Since $f(T,\mu_0)= -(2r-1)\log(2) <0,$ $(T,K^G,\mu_\epsilon,\alpha)$ is a Markov process with negative $f$-invariant for all $\epsilon\ge 0$ sufficiently small. 
\end{proof}

 In [Bo09] it is shown that no Bernoulli shift factors onto a shift with negative $f$-invariant. Hence each system constructed above is not even weakly isomorphic to a Bernoulli shift. It is interesting to compare this with the well-known result [FO70] that every mixing Markov chain over the integers is isomorphic to a Bernoulli shift. By comparison, it can be proven that for $\epsilon \in (0,1)$, the systems constructed above are uniformly mixing. This leads to an interesting open problem: classify mixing Markov systems over a free group up to measure-conjugacy.

\section{Markov approximations and the proof that $f=f_*$}\label{sec:f=f_*}

The purpose of this section is to prove:
\begin{thm}\label{thm:f_*=f}
Let $(T,X,\mu,\alpha)$ be a $G$-process with $H(\alpha)<\infty$. Let $\beta$ be a partition of $X$ with $H(\beta)<\infty$ and $\beta^G \subset \alpha^G$. Then  $f_*(\alpha|\beta^G)=f(\alpha|\beta^G)$.
\end{thm}
I do not know if the result holds if $H(\beta)=+\infty$. The proof is an application of theorem \ref{thm:f=F}. We will approximate the given process by a sequence of Markov processes. The first step is to embed the given process into a symbolic process as defined next.

\begin{defn}
A process of the form $(T,K^G,\mu,\alpha)$ where $T$ is the canonical action on $K^G$, $K$ is finite or countably infinite and $\alpha$ is the canonical partition is a {\bf symbolic process}.
\end{defn}

\begin{lem}\label{lem:embedding}
Let $(S,Y,\nu,\beta)$ be a $G$-process. If $\beta$ is generating then there is a canonical process isomorphism $\phi: (S,Y,\nu,\beta) \to (T,\beta^G,\mu,\alpha)$ where $ (T,\beta^G,\mu,\alpha)$ is symbolic.
\end{lem}

\begin{proof}
For $y\in Y$ define $\phi(y):G \to \beta$ by $\phi(y)(g) = B$ if $S_gy \in B\in\beta$. Let $T$ be the canonical action of $G$ on $\beta^G$.

If $f \in G$ then $\phi(S_gy)(f)=B $ iff $S_fS_gy \in B$ iff $S_{fg} y\in B$ iff $\phi(y)(fg)=B$ iff $(T_g\phi(y))(f)=B$. So $\phi$ is $G$-equivariant. If $B \in \beta$ then $\phi(B)=\{x \in \beta^G~|~ x(e)=B\}$. Thus $\phi$ maps $\beta$ to the canonical partition of $\beta^G$. Let $\mu=\phi_*(\nu)$. $\phi$ is invertible because $\beta$ is generating.
\end{proof}

\begin{lem}\label{lem:nstep}
Let $(T,K^G,\mu,\alpha)$ be a symbolic process. Let $\beta$ be a partition with $\alpha \le \beta \le \alpha^n$ for some $n\ge 0$. Then there exists a unique measure $\mu_\beta$ such that $(T,K^G,\mu_\beta,\beta)$ is Markov and
$$d_1\big( (T,K^G,\mu,\beta), (T,K^G,\mu_\beta,\beta) \big) = 0.$$
\end{lem}

\begin{proof}
By the previous lemma applied to $(T,K^G,\mu,\beta)$, there is a canonical $G$-equivariant embedding $\phi:K^G \to \beta^G$. Let $\{U_g\}_{g\in G}$ denote the canonical action of $G$ on $\beta^G$ and let $\gamma$ denote the canonical partition of $\beta^G$. Consider the process $\big(U, \beta^G, \phi_*\mu, \gamma\big)$. It is isomorphic to the process $(T,K^G,\mu,\beta)$.

By theorem \ref{thm:existence} there exists a unique measure $\nu$ on $\beta^G$ such that $\big(U, \beta^G, \nu, \gamma\big)$ is Markov and
$$d_1\Big( \big(U, \beta^G, \phi_*\mu, \gamma\big), \big(U, \beta^G, \nu, \gamma\big)\Big)=0.$$

Let $\mu_\beta$ be the pullback $\phi^*(\nu)$. It follows from the fact that $\alpha \le \beta\le \alpha^n$ that the support of $\nu$ is contained in the image of $\phi$. So $\mu_\beta$ is a well-defined $G$-invariant probability measure. In fact, $(T,K^G,\mu_\beta,\beta)$ is process-isomorphic (via $\phi$) to $(U,\beta^G,\nu,\gamma)$. So $(T,K^G,\mu_\beta,\beta)$ is a Markov process. It is easy to check that
$$d_1\big( (T,K^G,\mu,\beta), (T,K^G,\mu_\beta,\beta) \big) = 0.$$
\end{proof}

\begin{defn}\label{defn:approx}
If $(T,K^G,\mu,\alpha)$, $\beta$ and $\mu_\beta$ are as in the previous lemma then $\mu_\beta$ is called the {\bf Markov approximation to $\mu$} induced by $\beta$.
\end{defn}

\begin{lem}\label{lem:approx}
Let $(T,K^G,\mu,\alpha)$ be a symbolic process. Let $\{\beta_n\}_{n=1}^\infty$ be a sequence of partitions such that for all $n$ there exists integers $I(n) \le J(n)$ with $\alpha^{I(n)} \le \beta_n \le \alpha^{J(n)}$ and $\lim_{n\to\infty} I(n) = \infty$. Then $\mu_{\beta_n}$ converges to $\mu$ in the weak* topology.
\end{lem}

\begin{proof}
Since
$$d_1\big( (T,K^G,\mu_{\beta_n},\beta_n), (T,K^G,\mu,\beta_n) \big) = 0,$$
$\mu_{\beta_n}(B)=\mu(B) \forall B \in \beta_n$. Hence $\mu_{\beta_n}(B) = \mu (B) \forall B \in \alpha^{I(n)}$. Since $\lim_{n\to\infty} I(n)=+\infty$, this implies the lemma.
\end{proof}

Before proving theorem \ref{thm:f_*=f} we need to prove that $f$ and $f_*$ are upper semi-continuous in the variable $\mu$. As in \S \ref{sec:notation}, let $M(K^G)$ denote the space of all invariant Borel probability measures on $K^G$ where $K$ is finite or countable. If $\mu \in M(K^G)$ and $\beta$ is a partition of $K^G$, let $f(\mu,\beta)$ be the $f$-invariant of the process $(T,K^G,\mu,\beta)$ where $T$ is the canonical action.

\begin{lem}\label{lem:mucontinuous}
Let $\alpha$ be the canonical partition of $K^G$. Let $\sF$ be a $T(G)$-invariant Borel $\sigma$-algebra. Then the map $\mu \mapsto f_*(\mu,\alpha|\sF)$ is upper semi-continuous with respect to the weak* topology. Similarly,  the function $\mu \mapsto f(\mu,\alpha|\sF)$ is upper semi-continuous with respect to the weak* topology.
\end{lem}

\begin{proof}
It is well-known that for every $s\in S$, the function $\mu \mapsto h(T_s, \mu,\alpha|\sF)$ is upper semi-semicontinuous in the variable $\mu$ (e.g., [Gl03, lemma 15.1, page 270]). For example, this follows from the fact that, for every $n$, the function $\mu \mapsto \frac{1}{n+1}H(\mu, \bigvee_{k=0}^n T_s^{-k}\alpha|\sF)$ is continuous (since conditional expectation with respect to $\sF$ is continuous) and $h(T_s, \mu, \alpha|\sF)$ is the infimum of these functions. Thus, for every $n$, the function $\mu \mapsto F_*(\mu,\alpha^n|\sF)$ is upper semi-continuous. Since $f_*(\mu,\alpha|\sF)=\inf_n F_*(\mu,\alpha^n|\sF)$, the lemma follows. The proof for $f$ in place of $f_*$ is similar.
\end{proof}

\begin{proof}[Proof of theorem \ref{thm:f_*=f}]
 After replacing $\alpha$ with $\alpha \vee \beta$ if necessary, we may assume that $\alpha$ refines $\beta$. We may also assume that $\alpha$ is generating. So after applying the canonical embedding (lemma \ref{lem:embedding}), we may assume that $X=K^G$ and $\alpha$ is the canonical partition of $K^G$.

For each $n$, let $\mu_n=\mu_{\alpha^n}$ be the Markov approximation to $\mu$ induced by $\alpha^n$. We claim that
\begin{eqnarray*}
f(\mu,\alpha|\beta^G) &=& \lim_n F(\mu,\alpha^n|\beta^G) = \lim_n F(\mu_n,\alpha^n|\beta^G) = \lim_n F_*(\mu_n, \alpha^n|\beta^G)\\
 &=& \lim_n f_*(\mu_n,\alpha|\beta^G)\le  f_*(\mu,\alpha|\beta^G).
\end{eqnarray*}
The first equality holds by definition of $f$, the second holds since  $d_1\big( (T,K^G,\mu_{n},\alpha^n), (T,K^G,\mu,\alpha^n) \big) = 0$. The third and fourth equalities follow from theorem \ref{thm:f=F}. The previous lemma and lemma \ref{lem:approx} imply the last inequality.

For the reverse note that for any $s\in S$ and any $n\ge 0$,
\begin{eqnarray*}
h(T_s, \alpha^n|\beta^G) &=&\lim_{m\to\infty} H\Big(T^{-m-1}_s \alpha^n|\bigvee_{i=0}^m T^{-i}_s \alpha^n \vee \beta^G\Big)\\
&\le& H(T_s^{-1}\alpha^n|\alpha^n \vee \beta^G)\\
    &=& H(\alpha^n \vee T_s^{-1}\alpha^n|\beta^G) - H(\mu,\alpha^n|\beta^G).
    \end{eqnarray*}
Thus $F(\alpha^n|\beta^G) \ge F_*(\alpha^n|\beta^G)$. Take the limit as $n\to \infty$ to obtain $f(\mu,\alpha|\beta^G) \ge f_*(\mu,\alpha|\beta^G)$.
\end{proof}

\section{The Abramov-Rohlin Formula}\label{sec:abramov}

We can now prove theorem \ref{thm:abramov}.
\begin{proof}[Proof of theorem \ref{thm:abramov}]
By theorem \ref{thm:f_*=f}, it suffices to prove that $ f_*(\alpha|\beta^G) = f_*(\alpha \vee \beta) - f_*(\beta).$
The classical Abramov-Rohlin formula implies that if $n,m \ge 0$ and $s\in S$ and if $\beta^m_s$ denotes the smallest $T_s$-invariant $\sigma$-algebra containing $\beta^m$ then
$$h(T_s,\alpha^n|\beta^m_s) = h(T_s,\alpha^n \vee \beta^m)-h(T_s,\beta^m).$$
The definition of $F_*$ now implies $ F_*(\alpha^n |\beta^m) = F_*(\alpha^n \vee \beta^m) - F_*(\beta^m).$ Thus,
\begin{eqnarray*}
f_*(\alpha|\beta^G) &=& \lim_{n\to\infty} \lim_{m\to\infty} F_*(\alpha^n |\beta^m)=  \lim_{n\to\infty} \lim_{m\to\infty}  F_*(\alpha^n \vee \beta^m) - F_*(\beta^m)= f_*(\alpha\vee \beta)- f_*(\beta).
\end{eqnarray*}
The last equality follows from the fact that $F_*$ is monotone decreasing under splittings (proposition \ref{prop:splittingmonotone}) and proposition \ref{prop:splitting}.
\end{proof}

\section{A characterization of Markov processes}\label{sec:characterization}

The purpose of this section is to prove:
\begin{thm}\label{thm:characterization}
A $G$-process $(T,X,\mu,\alpha)$ is Markov if and only if $F(\alpha)=f(\alpha)$.
\end{thm}
This theorem is not used in the rest of the paper. 

\begin{cor}[Markov processes maximize the $f$-invariant]
Let $K$ be finite or countable and let $\mu$ be an invariant Borel probability measure on $K^G$ (with respect to the canonical action). If $\alpha$ is the canonical partition of $K^G$ then $f(T,\mu) \le F(T,\mu)$ with equality if and only if $(T,K^G,\mu,\alpha)$ is Markov.
\end{cor}
\begin{proof}
This follows from the theorem above and the fact that $f(T,\mu) \le F(T,\mu)$ always holds by definition of $f$ (see definition \ref{defn:f}).
\end{proof}

\begin{defn}
Let $(T,X,\mu,\alpha)$ be a $G$-process. If $Q \subset G$ is finite then let
$$\alpha^Q:=\bigvee_{q \in Q} T_q^{-1}\alpha.$$
\end{defn}

\begin{proof}[Proof of theorem \ref{thm:characterization}]
By theorem \ref{thm:f=F}, it suffices to prove that if $f(\mu,\alpha)=F(\mu,\alpha)$ then $(T,X,\mu,\alpha)$ is Markov. By lemma \ref{lem:embedding}, we may assume without loss of generality that $(T,X,\mu,\alpha)=(T,K^G,\mu,\alpha)$ is a symbolic process. By theorem \ref{thm:existence} there exists a Borel probability measure $\omega$ on $K^G$ such that $(T,K^G,\omega,\alpha)$ is Markov and $d_1\big((T,K^G,\mu,\alpha), (T,K^G,\omega,\alpha)\big)=0$.

{\bf Claim 1:} Let $Q \subset G$ be finite, right-connected and $e\in Q$. If for some $t \in S$, $A \in \alpha^{Q \cup Qt}$ then $\mu(A)=\omega(A)$.

Note that the claim implies the theorem, because it implies that $\mu(A)=\omega(A)$ for all $A \in \alpha^n$ for any $n\ge 0$ and thus $\mu=\omega$.

The claim is proven by induction on $|Q|$. If $|Q|=1$ then it follows from $d_1\big((T,K^G,\mu,\alpha), (T,K^G,\omega,\alpha)\big)=0$. So suppose $|Q| >1$. Then there exists $u \in S$ and a set $P \subset Q$ such that $P$ is right-connected, $e\in P$,  $|P|<|Q|$ and $Q \subset P \cup Pu$. The induction hypothesis implies that $\mu(A)=\omega(A)$ for all $A \in \alpha^{P \cup Ps}$ for any $s \in S$.

Note that $Q \cup Qt \subset P \cup Pu \cup Pt \cup Put$. Hence it suffices to show that $\mu(A)=\omega(A)$ for all $A \in \alpha^{ P \cup Pu \cup Pt \cup Put}$.

It suffices to show that for any $A, B, C, D \in \alpha^{P}$,
$$\mu(A \cap T_u^{-1}B \cap T_t^{-1}C \cap T_{ut}^{-1}D) = \omega( A \cap T_u^{-1}B \cap T_t^{-1}C \cap T_{ut}^{-1}D). $$
If $u=t$ and $B \ne C$ then both sides equal zero. If $u=t^{-1}$ and $A \ne D$ then both sides equal zero. So we may assume that these cases do not occur.

Note that
\begin{eqnarray}\label{eqn:mp5}
&&\mu(A \cap T_u^{-1}B \cap T_t^{-1}C \cap T_{ut}^{-1}D)\\
 &=&\mu(A \cap T_{t}^{-1}C) \mu(T_{ut}^{-1}D|A \cap T_t^{-1}C)\mu(T_u^{-1}B | A \cap T_t^{-1}C \cap T_{ut}^{-1}D)\\
&=&\omega(A \cap T_{t}^{-1}C) \mu(T_{ut}^{-1}D|A \cap T_t^{-1}C)\mu(T_u^{-1}B | A \cap T_t^{-1}C \cap T_{ut}^{-1}D).
\end{eqnarray}
The last line follows from the induction hypothesis. We will show that $\mu$ can be replaced with $\omega$ in the last line above. The next two claims help to reduce the problem.

\noindent {\bf Claim 2:} If
$$\mu(T_{ut}^{-1}D|A \cap T_t^{-1}C) = \mu(T_{ut}^{-1}D|T_t^{-1}C)$$
then
$$\mu(T_{ut}^{-1}D|A \cap T_t^{-1}C) = \omega(T_{ut}^{-1}D|A \cap T_t^{-1}C).$$
{\bf Claim 3:} If
$$\mu(T_u^{-1}B | A \cap T_t^{-1}C \cap T_{ut}^{-1}D) = \mu(T_u^{-1}B | A )$$
then
$$\mu(T_u^{-1}B | A \cap T_t^{-1}C \cap T_{ut}^{-1}D) =\omega(T_u^{-1}B | A \cap T_t^{-1}C \cap T_{ut}^{-1}D).$$

\noindent {\bf Proof of claim 2}. By lemmas \ref{lem:splitting1} and \ref{lem:splittings} $(T,K^G,\omega,\alpha^P)$ is Markov. Hence
\begin{eqnarray*}
\omega(T_{ut}^{-1}D|A \cap T_t^{-1}C) &=& \omega(T_{ut}^{-1}D|T_t^{-1}C)= \frac{\omega( T_{ut}^{-1}D \cap T_t^{-1}C) }{ \omega( T_t^{-1}C) }.
\end{eqnarray*}
By the induction hypothesis, $\omega(T_t^{-1}C) = \mu(T_t^{-1}C)$. By $G$-invariance and the induction hypothesis,
$$\omega( T_{ut}^{-1}D \cap T_t^{-1}C) = \omega(T_u^{-1}D \cap C) = \mu(T_u^{-1}D \cap C) = \mu( T_{ut}^{-1}D \cap T_t^{-1}C).$$
Hence the above implies
\begin{eqnarray}\label{eqn:mp4}
\omega(T_{ut}^{-1}D|A \cap T_t^{-1}C) &=& \frac{  \mu( T_{ut}^{-1}D \cap T_t^{-1}C) }{ \mu(T_t^{-1}C) }= \mu(T_{ut}^{-1}D|T_t^{-1}C)\\&=&	\mu(T_{ut}^{-1}D|A \cap T_t^{-1}C).
\end{eqnarray}
The last equality follows from the hypothesis of claim 2. This proves claim 2.

\noindent {\bf Proof of claim 3}. Since $(T,K^G,\omega,\alpha^P)$ is Markov,
\begin{eqnarray*}
&&\omega(T_u^{-1}B | A \cap T_t^{-1}C \cap T_{ut}^{-1}D) = \omega(T_u^{-1}B | A ) = \frac{  \omega(T_u^{-1}B \cap A ) }{ \omega(A) }\\
=&& \frac{  \mu(T_u^{-1}B \cap A ) }{ \mu(A) }= \mu(T_u^{-1}B | A ) = \mu(T_u^{-1}B | A \cap T_t^{-1}C \cap T_{ut}^{-1}D).
\end{eqnarray*}
The third equality uses the induction hypothesis. The last equality uses the hypothesis of claim 3. This proves claim 3.

Note that if $\mu(T_{ut}^{-1}D|A \cap T_t^{-1}C) = \omega(T_{ut}^{-1}D|A \cap T_t^{-1}C)$ and $\mu(T_u^{-1}B | A \cap T_t^{-1}C \cap T_{ut}^{-1}D) =\omega(T_u^{-1}B | A \cap T_t^{-1}C \cap T_{ut}^{-1}D)$ then equation \ref{eqn:mp5} implies
$$\mu(A \cap T_u^{-1}B \cap T_t^{-1}C \cap T_{ut}^{-1}D) = \omega(A \cap T_u^{-1}B \cap T_t^{-1}C \cap T_{ut}^{-1}D)$$
which implies the theorem.

If $u=t^{-1}$ then, by assumption, $A=D$. Hence $\mu(T_{ut}^{-1}D|A \cap T_t^{-1}C) = \omega(T_{ut}^{-1}D|A \cap T_t^{-1}C)$. If $u=t$ then, by assumption, $B=C$. Hence $\mu(T_u^{-1}B | A \cap T_t^{-1}C \cap T_{ut}^{-1}D) =\omega(T_u^{-1}B | A \cap T_t^{-1}C \cap T_{ut}^{-1}D).$

So by claims 2 and 3 it suffices to prove that if $u\ne t^{-1}$ then
$$\mu(T_{ut}^{-1}D|A \cap T_t^{-1}C) = \mu(T_{ut}^{-1}D|T_t^{-1}C)$$
and if $u \ne t$ then
$$\mu(T_u^{-1}B | A \cap T_t^{-1}C \cap T_{ut}^{-1}D) = \mu(T_u^{-1}B | A ).$$

By lemma \ref{lem:relative}, it suffices to prove the following claim.

\noindent {\bf Claim 4:} If $u\ne t^{-1}$ then
$$H(\alpha^{Put} | \alpha^{P \cup Pt} ) = H(\alpha^{Put} | \alpha^{Pt})=H(\alpha^{Pu}|\alpha^P) \textrm{ (by $G$-invariance)}$$
and if $u \ne t$ then
$$H(\alpha^{Pu} | \alpha^{P \cup Pt \cup Put}) = H(\alpha^{Pu} | \alpha^P ).$$
These entropies and all the ones below are with respect to $\mu$.


Both $P$ and $P \cup Pu$ are finite, right-connected and contain the identity element. Hence lemma \ref{lem:splittings}  implies $\alpha^{P }$ and $\alpha^{P \cup Pu}$ are splittings of $\alpha$. Proposition \ref{prop:splittingmonotone} implies
$$F(\alpha)=f(\alpha) \le F(\alpha^{P \cup Pu }) \le F(\alpha^{P}) \le F(\alpha).$$
So equality holds throughout. The above $F$ and $f$ values (and the ones below) are all with respect to $\mu$. Now,
\begin{eqnarray*}
0 &=& F(\alpha^{P \cup Pu}) - F(\alpha^{P}) \\
&=& (1-2r) H(\alpha^{ Pu} | \alpha^{P}) + \sum_{i=1}^r  H(\alpha^{ Pu \cup Pus_i} | \alpha^{P\cup Ps_i })  \\
&=& (1-2r) H(\alpha^{ Pu} | \alpha^{P}) + \sum_{i=1}^r H(\alpha^{ Pus_i} | \alpha^{P\cup Ps_i }) + H(\alpha^{ Pu } | \alpha^{P\cup Ps_i\cup Pus_i }).
\end{eqnarray*}
If, for some $i$, $u=s_i^{-1}$ then
$$H(\alpha^{ Pus_i} | \alpha^{P\cup Ps_i })  = 0.$$
If, for some $i$, $u=s_i$ then
$$H(\alpha^{ Pu } | \alpha^{P\cup Ps_i\cup Pus_i }) =0.$$
Hence one of the terms in the above sum equals zero. Since for every $i$
$$H(\alpha^{ Pus_i} | \alpha^{P\cup Ps_i }) \le  H(\alpha^{ Pu} | \alpha^{P})$$
and
$$H(\alpha^{ Pu } | \alpha^{P\cup Ps_i\cup Pus_i }) \le   H(\alpha^{ Pu} | \alpha^{P})$$
this implies that when $u \ne s_i^{-1}$,
$$H(\alpha^{ Pus_i} | \alpha^{P\cup Ps_i }) = H(\alpha^{ Pu} | \alpha^{P})$$
and when $u \ne s_i$,
$$H(\alpha^{ Pu } | \alpha^{P\cup Ps_i\cup Pus_i }) =   H(\alpha^{ Pu} | \alpha^{P}).$$
If, for some $i$, $s_i=t$ then the two equations above imply claim 4. Suppose instead that $s_j^{-1}=t$ for some $j$. By $G$-invariance,
\begin{eqnarray*}
F(\alpha) &=& (1-2r)H(\alpha) +\sum_{i=1}^r H(\alpha \vee T_{s_i}^{-1}\alpha)= (1-2r)H(\alpha) +\sum_{i=1}^r H(T_{s_i^{-1}}^{-1}\alpha \vee \alpha).
\end{eqnarray*}
Hence we may replace each $s_i$ in the proof of claim 4 with $s_i^{-1}$. This proves claim 4. As noted above, claim 4 implies claim 1 which implies the theorem.
\end{proof}

\section{Limits of Partitions}\label{sec:limits}

\begin{defn}\label{defn:partitionlimits}
Let $G \c^T (X,\cB,\mu)$, $\sF \subset \cB$ be a sub-$\sigma$-algebra and $\{\beta_i\}_{i=1}^\infty$ be partitions of $X$. We will write $\lim_{i \to \infty} \beta_i = \sF$ if for every partition $\alpha \subset \sF$ with $H(\alpha)<\infty$,
$$\lim_{i \to \infty} H(\alpha|\beta_i) =0$$
and there exists a sequence of partitions $\{\gamma_i\}_{i=1}^\infty$ with $\gamma_i \subset \sF$ and $\lim_{i\to\infty} d(\gamma_i,\beta_i)=0$. Here $d(\cdot,\cdot)$ is the Rohlin distance (definition \ref{defn:rohlin}).
\end{defn}

The purpose of this section is to prove the proposition below which will be used in the proof of the addition formula (theorem \ref{thm:freeyuz}).

\begin{prop}\label{prop:partitionapproximation}
Let $(T, X,\mu,\alpha)$ be a $G$-process with $H(\alpha)<\infty$. Let $\{\beta_i\}_{i=1}^\infty$ be partitions of $X$ with $H(\beta_i)<\infty$ such that $\lim_{i\to\infty} \beta_i = \alpha^G$. Then
$$f(\alpha) = f_*(\alpha)=\lim_{i\to\infty} F_*(\beta_i)=\lim_{i\to\infty} F(\beta_i).$$
\end{prop}

Here is an application.
\begin{cor}
There does not exist a finite-entropy generating partition of the canonical action of $G$ on $([0,1]^G,\lambda^G)$ where $\lambda$ is Lebesgue measure on $[0,1]$.
\end{cor}
\begin{remark}
This result was proven first in [Bo08b] (by a different method). It is an open question whether it holds for all countable groups $G$.
\end{remark}
\begin{proof}
Let $\sigma_1 \le \sigma_2 \le \ldots$ be an increasing sequence of finite partitions of $[0,1]$ such that $\bigvee_{i=1}^\infty \sigma_i$ is the $\sigma$-algebra of all measurable sets (up to sets of measure zero). Let $\pi:[0,1]^G \to [0,1]$ be the evaluation map $\pi(x):=x(e)$. Let $\beta_i$ be the pullback partition $\beta_i:=\pi^*(\sigma_i)$. Let $\alpha_i = \beta_i^i$. It is an easy exercise to show that $\alpha_i$ converges to the full $\sigma$-algebra of all measurable sets of $[0,1]^G$. So, assuming that the system $G \c^T ([0,1]^G,\lambda^G)$ has a finite generating partition, it follows from proposition \ref{prop:partitionapproximation} that $f(T) = \lim_{i\to\infty} F(\alpha_i)$. We will show that the later limit equals $+\infty$ which contradicts the fact that the $f$-invariant is the infimum of a set of real numbers.

Since $(T,[0,1]^G,\lambda^G,\beta_i)$ is a Bernoulli process, it follows from a simple calculation (performed in [Bo08a]) that $f(\beta_i)=F(\beta_i)=H(\beta_i)$. Since $\alpha_i$ is a splitting of $\beta_i$, this implies $F(\alpha_i)=H(\beta_i)$. By definition, $H(\beta_i)=H(\sigma_i)$. So we have $F(\alpha_i)=H(\sigma_i)$. Obviously, $\lim_{i\to\infty} H(\sigma_i) = +\infty$.
\end{proof}

We will need two simple lemmas.
\begin{lem}\label{lem:Fdistance}
If $\alpha, \beta$ are any partitions of $X$ with $H(\alpha)+H(\beta)<\infty$ then
$$|F(\alpha)-F(\beta)| \le (4r-1)d(\alpha,\beta).$$
\end{lem}
\begin{proof}
This follows immediately from the fact that
$$|H(\alpha)-H(\beta)|\le |H(\alpha)-H(\alpha\vee\beta)|+|H(\alpha\vee\beta)-H(\beta)|= d(\alpha,\beta)$$
and for any $s\in S$,
$$|H(\alpha \vee T_s^{-1}\alpha) - H(\beta \vee T_s^{-1}\beta)| \le d(\alpha \vee T_s^{-1}\alpha, \beta \vee T_s^{-1}\beta) \le 2d(\alpha,\beta).$$
\end{proof}

\begin{lem}
Let $\alpha$, $\{\beta_i\}_{i=1}^\infty$ be as in proposition \ref{prop:partitionapproximation}. If $\{\gamma_i\}_{i=1}^\infty$ is a sequence of partitions with $\lim_{i\to\infty} d(\gamma_i,\beta_i) = 0$ then $\lim_{i\to\infty} \gamma_i = \alpha^G$.
\end{lem}

\begin{proof}
Let $\omega \le \alpha^G$ be any partition with $H(\omega)<\infty$. Then
\begin{eqnarray*}
H(\omega|\gamma_i) &=& H(\omega \vee \gamma_i) - H(\gamma_i)\\
&\le&  |H(\omega\vee \gamma_i) - H(\omega \vee \beta_i)| + |H(\omega \vee \beta_i) - H(\beta_i)| + |H(\beta_i) - H(\gamma_i)|\\
&\le& d(\omega \vee \gamma_i,\omega \vee \beta_i)+ H(\omega|\beta_i) + d(\gamma_i,\beta_i)\\
&\le& H(\omega|\beta_i)  + 2d(\gamma_i,\beta_i).
\end{eqnarray*}
The result now follows from the hypothesis that $\lim_{i\to\infty} \beta_i = \alpha^G$.
\end{proof}

\begin{proof}[Proof of proposition \ref{prop:partitionapproximation}]
Since $\lim_{i\to\infty} \beta_i = \alpha^G$, there exists partitions $\gamma_i \subset \alpha^G$ such that $d(\beta_i,\gamma_i) \to 0$. Since $\gamma_i \subset \alpha^G$, we can assume, without loss of generality, that $\gamma_i \le \alpha^{n(i)}$ for some $n(i) \in \N$. For every $m>0$, $H(\alpha^m|\beta_i) \to 0$ implies that $d(\alpha^m \vee \gamma_i, \beta_i) \to 0$ too. So there is a sequence $m(i)$ such that $\lim_{i\to\infty} m(i)=+\infty$ and $d(\alpha^{m(i)} \vee \gamma_i,\beta_i) \to 0$. After replacing $\gamma_i$ with $\alpha^{m(i)} \vee \gamma_i$ we may assume that $\alpha^{m(i)} \le \gamma_i \le \alpha^{n(i)}$. 




 Propositions \ref{prop:splitting} and \ref{prop:splittingmonotone} imply $F(\mu, \alpha^{n(i)}) \le F(\mu, \gamma_i)$. Thus
$$f(\mu,\alpha) = \inf_{n\to\infty} F(\mu,\alpha^n) \le \liminf_{i\to\infty} F(\mu,\gamma_i)\le \limsup_{i\to\infty} F(\mu,\gamma_i).$$
We claim that equality holds in the above equation. Since $\gamma_i \le \alpha^G$ for all $i$, to prove the claim we may assume that $\alpha$ is generating. By lemma \ref{lem:embedding}, we may assume that $X=K^G$ and $T$ and $\alpha$ are the canonical action and partition respectively.

Let $\mu_i$ be the Markov approximation to $\mu$ induced by $\gamma_i$ (definition \ref{defn:approx}). Since $\alpha^{m(i)} \le \gamma_i$ and $m(i)$ tends to infinity with $i$, lemma \ref{lem:approx} implies that $\mu_i$ tends to $\mu$ in the weak* topology. Since $f$ is upper semi-continuous in the $\mu$ variable (lemma \ref{lem:mucontinuous}) and $F(\mu,\gamma_i) = F(\mu_i,\gamma_i)=f(\mu_i,\alpha)$ (by theorem \ref{thm:f=F}), this implies
$$ \limsup_{i\to\infty} F(\mu, \gamma_i) = \limsup_{i\to\infty} f(\mu_i,\alpha) \le f(\mu,\alpha).$$
This proves the claim. Since $\lim_{i\to\infty} d(\gamma_i,\beta_i)=0$, lemma \ref{lem:Fdistance} implies $f(\mu,\alpha) = \lim_{n\to \infty} F(\mu,\beta_i).$ This proves the proposition in the case of $F$. The proof with $F_*$ replacing $F$ is similar.
 \end{proof}

\section{Yuzvinskii's Addition Formula}\label{sec:yuz}
In this section, we prove theorem \ref{thm:freeyuz}. The proof makes use of a generalization of a result due to R. K. Thomas [Th71] which itself is a generalization of Yuzvinskii's formula. To state it properly, we need some definitions.
\begin{defn}\label{defn:skewproduct}
Let $G=\langle s_1,\ldots,s_r\rangle$ and $G \c^T (X,\cB,\mu)$. Let $\Gamma$ be a separable compact group with Haar probability measure $\nu$. Let $\{U_g\}_{g\in G}$ be an action of $G$ on $\Gamma$ by homomorphisms that preserve Haar measure.

A {\bf cocycle} for the actions $T$ and $U$ is a measurable map $\phi: G \times X \to \Gamma$ satisfying
$$\phi(g_2g_1, x) = U_{g_2}\big(\phi(g_1,x)\big)\phi(g_2,T_{g_1}x) ~\forall g_1,g_2 \in G, x \in X.$$
The {\bf skew product action} $\{S_g\}_{g\in G}$ of $G$ on  $(X \times \Gamma, \mu \times \nu)$ is defined by
$$S_g(x,\gamma)=\big(T_gx, U_g(\gamma) \phi(g,x)\big) ~\forall g\in G, x \in X, \gamma \in \Gamma.$$
We also write $S=T\times_\phi U$.
\end{defn}

\begin{thm}\label{thm:thomas}
Let $T,U,S,\phi$, etc. be as in the previous definition. Suppose $\Gamma$ is either totally disconnected, a Lie group or a connected finite-dimensional abelian group. If there are finite-entropy generating partitions $\alpha, \beta$ for $G \c^T (X,\cB,\mu)$ and $G \c^U (\Gamma,\textrm{Haar}(\Gamma))$ respectively then
$$f(S) = f(T\times_\phi U)= f(T) + f(U).$$
\end{thm}

In [Th71], R. K. Thomas proved the above theorem in the case $G=\Z$ or $\N$ without the finite-entropy restriction and without the restrictions on $\Gamma$. His proof relies on ideas from [Yu65]. Next let us see how theorem \ref{thm:freeyuz} follows from theorem \ref{thm:thomas}.

\begin{proof}[Proof of theorem \ref{thm:freeyuz} assuming theorem \ref{thm:thomas}]
Let $\sigma:\sG/\cN \to \sG$ be a Borel cross-section (i.e., $\sigma(\gamma \cN) \in \gamma\cN$ for $\gamma \in\sG$). Define a cocycle $\phi:G \times (\sG/\cN) \to \cN$ by $\phi(g,\gamma \cN) = T_g(\sigma(\gamma \cN)) \sigma(T_g(\gamma)\cN)^{-1}$. Define $\psi: \sG/\cN \times \cN \to \sG$ by $\psi(\gamma \cN, k) = k\sigma(\gamma \cN)$. An elementary calculation shows that $\psi$ conjugates the skew-action $T_{\sG/\cN} \times_\phi T_\cN$ with the $T_\sG$. Now apply theorem \ref{thm:thomas}.
\end{proof}

\begin{defn}
A group $\Gamma$ is {\bf rigid} if there exists an increasing sequences $\xi_1 \le \xi_2 \le \ldots$ of finite partitions of $\Gamma$ and a real number $Q>0$ such that $H(\alpha|\xi_i) \to 0$ for all finite-entropy partitions $\alpha$ and $H(\xi_i \gamma | \xi_i) \le Q$ for all $i$ and all $\gamma\in \Gamma$.
\end{defn}

\begin{thm}[Th71, theorem 2.3]
Suppose $G$ is isomorphic to either $\Z$ or $\N$. Let $T, U, X, \Gamma, \phi$ be as in definition \ref{defn:skewproduct}. Suppose $\Gamma$ is rigid. Let $\alpha$ and $\beta$ be partitions of $X$ and $\Gamma$ respectively. Let $\alpha \times \beta$ denote the product partition on $X \times \Gamma$. Then
$$h(T \times_\phi U, \alpha \times \beta) = h(T,\alpha) + h(U,\beta).$$
\end{thm}
\begin{proof}
In theorem 2.3 of [Th71] this result is proven under the assumption that $\alpha$ and $\beta$ are generating partitions. However,  the proof yields this more general result with only minor obvious modifications.
\end{proof}


\begin{prop}\label{prop:rigid}
Theorem \ref{thm:thomas} is true whenever $\Gamma$ is rigid.
\end{prop}

\begin{proof}
Let $\{\alpha_n\}$ be a sequence of finite partitions of $X$ such that $\alpha_n \to \alpha^G$. Similarly, let $\{\beta_n\}$ be a sequence of finite partitions of $\Gamma$ such that $\beta_n \to \beta^G$. By proposition \ref{prop:partitionapproximation}
\begin{eqnarray*}
f(T \times_\phi U) &=& \lim_{n\to\infty} F_*(T \times_\phi U, \alpha_n \times \beta_n).
\end{eqnarray*}
The previous proposition and the definition of $F_*$ implies
$$ F_*(T \times_\phi U, \alpha_n \times \beta_n) = F_*(T,\alpha_n) + F_*(U,\beta_n)$$
for any $n$. Now take the limit as $n\to\infty$ and apply proposition \ref{prop:partitionapproximation} again to obtain
$$f(T \times_\phi U) = f_*(T) + f_*(U) = f(T)+f(U).$$
\end{proof}

\begin{prop}
Totally disconnected groups, compact Lie groups, and finite-dimensional compact connected abelian groups are rigid.
\end{prop}

\begin{proof}
Rigidity for totally disconnected groups and finite-dimensional connected abelian groups is shown in theorems 7.2 and 7.3 of [Yz65]. There is a minor error in the abelian case, reproduced in [Th71, theorem 2.6]. It is corrected in [LSW90, lemma B.5]. Compact Lie groups were proven to be rigid in [Th71, theorem 2.5].
\end{proof}

Theorem \ref{thm:thomas} now follows from the above and proposition \ref{prop:rigid}. I conjecture that theorem \ref{thm:thomas} (and therefore theorem \ref{thm:freeyuz}) holds for all compact separable groups $\Gamma$.



\end{document}